\documentclass[a4paper,11pt]{article}
\usepackage{graphicx}
\usepackage{amsmath}
\usepackage{amsfonts}
\usepackage{amssymb}
\usepackage{xcolor}
\usepackage{hyperref}

\newenvironment{proof}[1][Proof]{\noindent\textbf{#1.} }{\
\rule{0.5em}{0.5em}}

\newtheorem{theorem}{Theorem}

\newtheorem{corollary}[theorem]{Corollary}

\newtheorem{definition}[theorem]{Definition}
\newtheorem{example}[theorem]{Example}

\newtheorem{lemma}[theorem]{Lemma}

\newtheorem{proposition}[theorem]{Proposition}
\newtheorem{question}[theorem]{Question}
\newtheorem{remark}[theorem]{Remark}

\newcommand\A{\mathbb A}

\newcommand\C{\mathbb C}

\newcommand\N{\mathbb N}

\newcommand\R{\mathbb R}

\newcommand\Z{\mathbb Z}
\newcommand\x{\mathbf x}
\newcommand\y{\mathbf y}

\newcommand\CC{{\mathcal C}}

\newcommand\CP{{\mathcal P}}

\newcommand\CL{{\mathcal L}}

\newcommand\Tan{\operatorname{Tan}}
\newcommand\PTan{\operatorname{PTan}}

\def\square{\ \rule{0.5em}{0.5em}}

\title{On the geometric degree of the tangent bundle of a smooth algebraic variety
\footnote{Partially supported by Universidad de Buenos Aires (UBACYT 20020190100116BA) and  CONICET (PIP 2021-2023 GI 11220200101015CO), Argentina.}}
\author{Gabriela Jeronimo{$^{1, 2, 3}$}, Leonardo Lanciano{$^{1,}$}\footnote{Corresponding author (email: llanciano@dm.uba.ar)} \ and Pablo Solern\'o{$^{1,2}$}}

\date{}

\begin{document}

\maketitle

\begin{minipage}{14cm}
\noindent {\small $^1$ Universidad de Buenos Aires, Facultad de Ciencias Exactas y Naturales,  Departamento de Matem\'atica. Buenos Aires, Argentina.}

\noindent {\small $^2$  CONICET -- Universidad de Buenos Aires, Instituto de Investigaciones Matem\'aticas ``Luis A. Santal\'o'' (IMAS). Buenos Aires, Argentina.}

\noindent {\small $^3$ Universidad de Buenos Aires, Ciclo B\'asico Com\'un,  Departamento de Ciencias Exactas. Buenos Aires, Argentina.}
\end{minipage}

\bigskip

\begin{abstract}
We present bounds for the geometric degree of the tangent bundle and the tangential variety
of a smooth affine algebraic variety $V$ in terms of the geometric degree of $V$. 
We first analyze the case of \emph{curves}, showing an explicit relation between these  degrees. In addition, for parametric curves, we obtain upper bounds that are linear in the degree of the given curve. 
In the case of varieties of \emph{arbitrary dimension}, we prove general upper bounds for the degrees of the tangent bundle and the tangential variety of $V$ that are exponential in the dimension or co-dimension of $V$,
and a quadratic upper bound that holds for varieties defined by \emph{generic} polynomials.
Finally, we characterize the smooth irreducible varieties with a tangent bundle of minimal degree.
\end{abstract}

\section{Introduction}

The tangent bundle of a smooth (differential, algebraic, etc.) variety is one of the fundamental objects in geometry. A survey on properties of tangent bundles, including dimension and smoothness, in the general framework of schemes can be found in \cite{Kunz1999}. 
The present paper is concerned with studying further properties of tangent bundles of affine algebraic varieties.

An essential invariant related to the geometric complexity of an algebraic variety is its geometric degree, which in some sense measures how twisted the variety is. We refer the reader to \cite{Heintz1983} for a precise definition and basic properties of the geometric degree of arbitrary affine varieties. This invariant turns out to be a key parameter involved in the complexity of algorithmic procedures for polynomial system solving  (see, for instance, \cite{BaMu93,GIUSTI19971223,Castro2003TheHO}).

In this paper we present some estimations concerning the geometric degree of the tangent bundle of smooth affine algebraic varieties in terms of the degree of the given variety. This kind of questions appear naturally in order to analyze complexity aspects in symbolic approaches to the integrability of systems of differential algebraic equations following the Cartan-Kuranishi principle of prolongation-projection (see \cite{cartan1945systemes, kuranishi1957}), including the elimination of unknowns in differential algebraic systems (see, for instance, \cite{OPV2022}) and the effective differential Nullstellensatz (see, for instance, \cite{DJS2014,GKO2016}). They are also related to the study of the properties of the jet schemes associated to a given variety (see for example \cite{Musta2000JetSO,Sebag2017ARO}). 

For an arbitrary \emph{smooth irreducible algebraic variety} $V$  of dimension $d$ in the affine space $\A^n$ over an algebraically closed field $K$ (i.e. the space $K^n$ equipped with the Zariski topology), the tangent bundle $TV\subseteq \A^{2n}$ is a smooth irreducible algebraic variety of dimension $2d$ (see \cite{Kunz1999} or Proposition \ref{prop:dimTV} below). It consists of all pairs $(p,q)\in \A^{2n}$ 
such that $p\in V$ and $q\in \A^n$ is a vector which is tangent to $V$ at $p$ or the zero vector. A closely related algebraic variety is the \emph{tangential variety} of $V$ (see \cite[Example 15.4]{harris1992algebraic}):
if $\pi_2:\A^{2n}\to \A^n$ is the projection which forgets the first $n$ coordinates, the tangential variety of $V$, denoted by $\Tan(V)$, is the Zariski closure of $\pi_2(TV)\subseteq \A^{n}$, i.e. the closure of the set of all those vectors which are tangent at some point of $V$. Clearly $\Tan(V)$ is irreducible (since $TV$ is) and its dimension is bounded by $\min\{2d,n\}$. However, unlike the case of the tangent bundle, in general it is a hard problem to determine the basic properties of the tangential variety of a given variety $V$, such as its ideal, dimension, etc. For example, in \cite{oeding}, the ideal and the ring of coordinates for the tangential varieties of Segre-Veronese varieties are computed and in \cite{cattaneo}, the dimension and degree of the tangential varieties of projective surfaces with some additional properties about its embedding are determined.
For the particular case of an irreducible smooth curve $\CC\subseteq \A^n$, it can be seen that $\dim(\Tan(\CC))=2$ if $\CC$ is not a line (see \cite[Exercise 8.1.2]{landsberg} or Lemma \ref{lem:tanC} below).

Concerning the geometric degree, the inequality
$\deg(\Tan(V))\le \deg(TV)$ (see Proposition \ref{prop:degTanVledegTV}) 
implies that the computation of (upper bounds for) the degree of $TV$ also provides upper bounds for the degree of $\Tan(V)$. 

We start by considering the case of a smooth irreducible algebraic curve $\CC\subseteq \A^n$. Our first main result is an explicit relation between the geometric degrees of $T\CC$ and $\Tan(\CC)$: 

\bigskip
\noindent \textbf{Theorem A.}
{\emph{Let $\CC \subseteq  \mathbb{A}^n$ be a smooth irreducible algebraic curve. Then,
\[        \deg(T\CC) = \deg(\CC) + \omega(\CC) \deg(\operatorname{Tan}(\CC)),\]
where $\omega(\mathcal{C})$ is the cardinality of the typical fiber of $\pi_2: \A^{2n} \to \A^n$ which forgets the first $n$ coordinates (roughly speaking, $\omega(\CC)$ is the number of points in $\mathcal{C}$ having the same generic tangent direction).}}

\bigskip

As a consequence of this result, we prove in Theorem \ref{thm:linear} that if $V$ is a smooth irreducible variety of \emph{arbitrary} dimension such that $\deg(TV)=\deg(V)$, then $V$ is necessarily a linear variety.

We also study the degree of the tangent bundle for the particular case of parameterized curves. We prove that if a curve $\CC$ has a parametric \emph{polynomial} representation satisfying certain assumptions (see Section \ref{subsection: parametric algebraic curves}), then $\deg(T\CC)=2\deg(\CC)-1$, whereas if it has a parametrization given by \emph{rational} functions, under the same assumptions, the bound $\deg(T\CC)\le 3\deg(\CC)-2$ holds (see Theorems \ref{thm:degTCpolyparam} and \ref{thm:degTCratparam} respectively). Note that in both cases the bound for the geometric degree of the tangent bundle is linear in the geometric degree of the curve. However, this is far from being the typical situation: for an  algebraic variety $V$ defined by \emph{generic} polynomials, the inequality  $\deg(TV)\le (\deg(V))^2$ holds (see Corollary \ref{corolario cuadrado}). 

For an \emph{arbitrary} smooth irreducible variety, we prove an upper bound for $\deg(TV)$ and \emph{a fortiori} for $\deg(\Tan(V))$:

\bigskip
\noindent \textbf{Theorem B.}
{\emph{Let $V\subseteq \A^n$ be a smooth irreducible algebraic variety of dimension $d$. Then,
\[ \deg(TV) \leq \min \left \{(\deg(V))^{n-d+1}, \deg(V)\left((n-d)(\deg(V)-1)+1\right)^{d} \right\}.\]
In particular, the same upper bound holds for $\deg(\Tan(V))$.}}

\bigskip

Nevertheless, we conjecture that an upper bound for $\deg(TV)$ depending quadratically on $\deg(V)$ should be valid (see Question \ref{conj: quadratic}). Observe that our general bound applied in the case of hypersurfaces ($d=n-1$) or curves ($d=1$) and the bound valid in the case of generic varieties give evidence for this presumption.

\bigskip

The paper is organized as follows: in Section \ref{sec:preliminaries}, we introduce the basic notions and notation, and point out the previous results we will use along the paper.
Section \ref{sec:smooth} is devoted to establishing some basic facts (dimension, smoothness, etc.) about the tangent bundle and the tangential variety associated to a smooth algebraic variety. In Section \ref{sec:curves} we consider the tangent bundle of smooth algebraic \emph{curves}:  we prove Theorem A and the bounds for the degree of the tangent bundle of parameterized curves (Theorems \ref{thm:degTCpolyparam} and \ref{thm:degTCratparam}, respectively). Finally, in Section \ref{sec:bounds}, we prove Theorem B and we show that the trivial lower bound $\deg(V)$ for the degree of the tangent bundle is attained only in the case of linear varieties (Theorem \ref{thm:linear}).

\section{Preliminaries} \label{sec:preliminaries}

Let $K$ be an algebraically closed field of characteristic $0$ and $K[\textbf{x}]$ the polynomial ring in $n$ variables $\textbf{x}:=x_1,\ldots,x_n$. We denote $\A^n$ the affine space $K^n$ equipped with the Zariski topology. For an algebraic variety $V\subseteq \A^n$, the ideal of all polynomials in $K[\textbf{x}]$ which vanish over $V$ is denoted by $I(V)$. For any subset $F\subseteq K[\x]$ we write $V(F)$ for the algebraic variety of the common zeroes in $\A^n$ of all the polynomials in $F$.

\subsection{Geometric degree, B\'ezout inequalities and Bernstein-Kushnirenko's theorem} \label{section:bezout}

Here we recall briefly the definition of the geometric degree of an algebraic variety and its first properties  (mainly the B\'ezout Inequality, see Proposition \ref{prop:bezout} below). For the proofs and complete statements see, for instance, \cite{mumford1976algebraic,harris1992algebraic} in the case of projective varieties or \cite{Heintz1983} for affine (and constructible) sets.

Let $V\subseteq\A^n$ be an irreducible algebraic variety of dimension $d$. Then, for almost all linear variety $L$ of dimension $n-d$, the intersection $V\cap L$ is a finite set. The cardinality of this set is called the \emph{geometric degree} of $V$ and denoted by $\deg(V)$. Let us remark that $\deg(V)$ is also the maximal (finite) cardinality obtained by intersections with linear varieties of complementary dimension (see for instance \cite[Lemma 1 and Prop.~1]{Heintz1983}).
It is easy to see that the degree of any linear variety is $1$ and the degree of a hypersurface agrees with the degree of the square-free equation defining it.

If the variety $V$ is not necessarily irreducible, we extend this notion by defining its degree as the sum of the degrees of \emph{all} its irreducible components. Even if this extension lost the geometric flavor of the previous definition in the non-equidimensional case, it is a quite natural generalization which allows to state and prove the following general B\'ezout Inequalities in the affine case (see \cite[Theorem 1]{Heintz1983}, 
and \cite[Proposition 2.3]{HS82}):

\begin{proposition} \label{prop:bezout}
Let $V,V_1,\ldots,V_r\subseteq\A^n$ be algebraic varieties. Then \[\deg(V\cap V_1\cap\cdots \cap V_r)\le \deg(V)\min \Big\{ \prod_{i=1}^r \deg(V_i)\ ,\ \big(\max_{1\le i\le r} \deg(V_i)\big)^{\dim(V)}\Big\}.
\]
In particular, if $f_1,\ldots,f_r\in K[\x]$, we have the inequality
\[\deg(\{x\in \A^n\mid f_1(x)=0,\ldots,f_r(x)=0\})\le \min \Big\{ \prod_{i=1}^r \deg f_i\ ,\ (\max_{1\le i\le r} \deg f_i)^n.\Big\}\ \square\]
\end{proposition}

For each $\alpha:=(\alpha_1,\ldots,\alpha_n)\in \mathbb{N}_0^n$ we denote by $\x^\alpha$  the monomial $x_1^{\alpha_1}\ldots x_n^{\alpha_n}$. Let $f\in K[\x]$ be a polynomial and write $f=\sum_{\alpha\in A} c_\alpha \x^\alpha$ where $c_\alpha\ne 0$ for all $\alpha\in A$. The set $A\subseteq \mathbb{N}_0^n$ is called the \emph{support of $f$} and its convex hull in $\R^n$ is the  \emph{Newton polytope of $f$}, denoted by $\mathcal{N}(f)$.

Let $P_1,\ldots,P_n$ be polytopes in $\R^n$. The \emph{mixed volume $MV(P_1,\ldots,P_n)$ of $P_1,\ldots,P_n$} is defined as the alternating sum \[ MV(P_1,\ldots,P_n)=\sum_{k=1}^{n}(-1)^{n-k} \sum_{I\subseteq\{1,\ldots,n\},\, |I|=k} \textrm{Vol}_{\R^n}(\sum_{i\in I}P_i),\]
where $\textrm{Vol}_{\R^n}$ is the usual volume in $\R^n$ and $\sum_{i\in I}P_i:=\{\sum_{i\in I}x_i\mid x_i\in P_i \ \forall i\in I\}$ (see e.g 
\cite[Chapter 7, \S4]{CoxLittleOshea2006usingAlgebraicGeometry}). In the particular case when $P_1=\dots=P_n=P$, we have that $MV(P,\dots, P) = n! \textrm{Vol}_{\R^n}(P)$.

Now, let $v\in \mathbb{Q}^n\setminus\{0\} $ be an arbitrary vector and $A \subseteq \mathbb{N}_{0}^n$ a finite set. We denote $m_v(A) := \min\{\alpha \cdot v\ |\ \alpha\in A\}$ and $A_v := \{\alpha \in A\ | \ \alpha \cdot v = m_v(A)\}$ (the dot ``$\cdot$" denotes the usual inner product in $\R^n$). For a polynomial $f\in K[\x]$ with support $A\subseteq \mathbb{N}_{0}^n$ we set  $f_v:=\sum_{\alpha\in A_v} c_\alpha \x^\alpha$.

With this notation we can state the so-called Bernstein-Kushnirenko's theorem (see \cite{bernstein1975number,kushnirenko1976polyedres,CoxLittleOshea2006usingAlgebraicGeometry}):
\begin{proposition}\label{prop:BKK}
Let $f_1,\ldots,f_n\in K[\x]$ and let $\delta$ be the number of their isolated common zeros in $(K\setminus \{0\})^n$ counted according to their multiplicities. Then $\delta \le MV(\mathcal{N}(f_1),\ldots,\mathcal{N}(f_n))$ and, if the supports of $f_1,\ldots,f_n$ are previously fixed, the equality holds generically. Moreover:
\begin{itemize}
\item If for all $v\in \mathbb{Q}^n\setminus\{0\}$ the polynomials $f_{1v},\ldots,f_{nv}$ have no common zeros in \linebreak $(K\setminus \{0\})^n$, then $\delta = MV(\mathcal{N}(f_1),\ldots,\mathcal{N}(f_n))$ and all the common zeros of $f_1,\dots, f_n$ in $(K\setminus \{0\})^n$ are isolated.
\item If for some $v\in \mathbb{Q}^n\setminus\{0\}$ the polynomials $f_{1v},\ldots,f_{nv}$ have a common zero in \linebreak  $(K\setminus \{0\})^n$, then $\delta < MV(\mathcal{N}(f_1),\ldots,\mathcal{N}(f_n))$  if   $MV(\mathcal{N}(f_1),\ldots,\mathcal{N}(f_n))\ne 0$ and $\delta=0$ otherwise.  $\square$
\end{itemize}
\end{proposition}

\subsection{Smooth algebraic varieties}

A variety $V$ is called \emph{smooth at a point $p\in V$} if and only if its local ring $\mathcal{O}_p$ is a regular local ring, i.e.~$\dim_{\operatorname{Krull}} ({\mathcal O}_p)=\dim_K (\mathfrak{m}/\mathfrak{m}^2)$, where $\mathfrak{m}$ is the maximal ideal of $\mathcal{O}_p$. In other words, the local dimension of $V$ at $p$ is the dimension of the cotangent space. See, for instance, \cite{Kunz2013,eisenbud1995commutative} for basic facts about regular rings and smooth varieties. 

A variety $V$ is called \emph{smooth} if it is smooth at every point. Since regular rings are in particular domains, a smooth variety $V$ is locally irreducible and then, all its connected components are smooth too and coincide with the irreducible components of $V$. 

A foundational result concerning smooth varieties is the classical \emph{Jacobian Criterion} (see \cite[Ch.VI, \S 1., Proposition 1.5]{Kunz2013} and \cite[\S.29]{matsumura1980commutative}):

\begin{proposition}\label{prop: jacobian}
Let $V\subseteq\A^n$ be an algebraic variety, $I(V)\subseteq K[\x]$ its ideal and $p\in V$. 
\begin{enumerate}
\item If $f_1,\ldots,f_r\in I(V)$, then the rank of the $r\times n$ Jacobian matrix $\dfrac{\partial (f_1,\ldots,f_r)}{\partial (x_1,\ldots,x_n)}(p )$ is at most $n-\dim_{\operatorname{Krull}}(\mathcal{O}_p)=n-\dim _p(V)$ (where $\dim _p(V)$ denotes the local dimension of the variety $V$ around the point $p$).
\item The following statements are equivalent:
\begin{enumerate}
\item V is smooth at $p$.
 \item If $I(V)=(f_1,\ldots,f_r)$, then 
 $\operatorname{rank} \dfrac{\partial (f_1,\ldots,f_r)}{\partial (x_1,\ldots,x_n)}(p )=n-\dim_p(V)$.
 \item There exist $f_1,\ldots,f_r\in I(V)$ such that $\operatorname{rank} \dfrac{\partial (f_1,\ldots,f_r)}{\partial (x_1,\ldots,x_n)}(p )=n-\dim_p(V)$.
\end{enumerate}
Moreover, if $f_1,\ldots,f_r\in I(V)$ satisfy condition (c) for every $p\in V$, then $I(V)=(f_1,\ldots,f_r)$. In particular a variety $V$ smooth at a point $p$ is locally ideal theoretically complete intersection in a neighborhood of $p$. $\square$
\end{enumerate}
\end{proposition}

Given an \textit{arbitrary} algebraic irreducible variety $V$, we denote by $\text{Reg}(V)$ the set of all the points $p\in V$ such that $V$ is smooth at $p$. It is easy to see by means of the Jacobian Criterion that $\textrm{Reg}(V)$ is a Zariski dense open subset of $V$.

We now recall a remarkable result which states the existence of a system of generators for the ideal of a smooth variety whose degrees are bounded by $\deg(V)$:

\begin{proposition} \label{prop:mumford}
Let $V \subseteq \mathbb{A}^n$ be a smooth algebraic variety. There exist polynomials $f_1,\ldots,f_r \in I(V)$ such that
$I(V) = (f_1,\ldots,f_r)$  and $\deg(f_i) \leq \deg(V)$ for $ i=1, \ldots, r$.
\end{proposition}

\begin{proof}
See for instance \cite[Theorem 1]{Mumfordvarieties}, \cite[Section 2]{seidenberg1975} or  \cite[Section 2]{catanese92}.   
\end{proof}

\subsection{On the cardinality of a zero-dimensional fiber}

In this section we discuss more or less folkloric results concerning the number of points in zero-dimensional fibers of morphisms between algebraic varieties that we need in the sequel.

Let $V\subseteq \A^n$ and $W\subseteq \A^m$ be irreducible algebraic varieties and $f:V\to W$ a morphism. 
We say that $f$ is \emph{dominant} if the image of $f$ is dense in $W$ or, equivalently if the induced ring morphism $f^*:K[W]\to K[V]$ is injective,  that $f$ is \emph{finite} if $f^*$ is integral, and that $f$ is  \emph{quasi-finite} if for every $y$ in the image of $f$, the fiber $f^{-1}(y)$ is a finite set. It is easy to see that finite implies quasi-finite (see for instance \cite[Ch.~1, Section 5.3]{shafarevich2013basic}). Note that from the Theorem on the dimension of fibers (see for instance \cite[Ch.~1, Section 6.3, Theorem 1.25]{shafarevich2013basic}), if $f$ is dominant and quasi-finite, necessarily $\dim(V)=\dim(W)$; in particular, there is a natural field inclusion $K(W)\hookrightarrow K(V)$ which is algebraic and finite.

\begin{proposition}\label{prop:fibersNormal}
Let $f:V\to W$ be a dominant quasi-finite morphism of irreducible affine algebraic varieties. Suppose that $W$ is normal (i.e.~$K[W]$ is integrally closed in its fraction field). Then, for every $y$ in the image of $f$ we have that $ \# f^{-1}(y) \le [K(V):K(W)].$ 
Moreover, the equality holds in a non-empty Zariski open subset of $W$.
\end{proposition}

\begin{proof}
See \cite[Theorem 66.5]{musili}.
\end{proof}

\bigskip

We notice that if $W$ is not assumed to be normal, the inequality in the above theorem may not be true for \emph{all} the fibers; however, the existence of a dense open set where the equality holds remains true. In this sense, we will use the following consequence of the previous result:

\begin{corollary}\label{coro:fibra}
Let $f:V \rightarrow W$ be a dominant morphism of irreducible affine varieties of the same dimension. There exists a non-empty Zariski open subset $\mathcal{U}\subseteq W$ such that for all $y\in \mathcal{U}$ the fiber $f^{-1}(y)$ is finite and $\# f^{-1}(y) = [K(V):K(W)]$ holds.
\end{corollary}

\begin{proof} Being $f$ dominant, Chevalley's Theorem (see for instance  
\cite[Chapter I, \S 5.3, Theorem 1.14]{shafarevich2013basic}) 
implies that there exists a non-empty Zariski open set which is contained in the image of $f$. On the other hand, from the Theorem on the dimension of fibers, the fiber of a generic point is zero-dimensional because both varieties $V$ and $W$ have the same dimension. Finally, since $\textrm{Reg}(W)$ contains a dense open subset and the affine open sets are a basis of the Zariski topology of $W$, we may suppose that there exists a non-empty open $\mathcal{U}=W\setminus \{h=0\}$ for a suitable $h\in K[W]$ such that the morphism $f:f^{-1}(\mathcal{U})\to \mathcal{U}$ is under the conditions of the previous theorem (observe that $\mathcal{U}$ is isomorphic to a smooth affine variety, hence normal). The corollary follows by observing that $K(W)=K(\mathcal{U})$ and $K(V)=K(f^{-1}(\mathcal{U}))$.
\end{proof}

\section{The tangent bundle and the tangential variety} \label{sec:smooth}

Let $V\subseteq\A^n$ be a variety of dimension $d$, smooth at a point $p\in V$. We define $T_pV$, the \emph{tangent space} of $V$ at $p$, as the dual of the $d$-dimensional cotangent vector space $\mathfrak{m}/\mathfrak{m^2}$.
From a less intrinsic point of view (see for instance \cite[Section 6.3]{reid_1988}), if $f_1,\ldots,f_r$ is a system of generators of $I(V)$ and $\nabla f_1,\ldots ,\nabla f_r$ are their gradient vectors, the tangent space $T_pV$ is the linear subspace of $\A^n$:
\[
 T_pV:=\{{y}\in\A^n\ |\ \nabla f_1(p)\cdot{y}=0,\ldots ,\nabla f_r(p)\cdot{y}=0\}
\]
where $\cdot$ denotes the usual scalar product in $K^n$.

\begin{definition}
 Let $V\subseteq \A^n$ be a smooth algebraic variety. The tangent bundle $TV\subseteq\A^{2n}$ is defined as the set:
 \[
  TV:=\{ (x,{y})\in\A^{2n}\ |\ x\in V,\ {y}\in T_xV\}.
 \]
\end{definition}

Clearly $TV$ is also an algebraic variety: if $I(V)=(f_1,\ldots,f_r)$ then $TV$ is defined by
\begin{equation}\label{eq:defTV}
 TV=\{(x,y)\in \A^{2n} \mid f_1({x})=0,\ldots,f_r({x})=0,\nabla f_1({x})\cdot {y}=0,\ldots ,\nabla f_r({x})\cdot {y}=0\}.
\end{equation}

It is easy to see that if $V$ is smooth and irreducible then $TV$ inherits both properties (see \cite{Kunz1999} for a more general result):

\begin{proposition} \label{prop:dimTV}
 Let $V\subseteq\A^n$ be a smooth irreducible variety of dimension $d$. Then, $TV\subseteq \A^{2n}$ is also a smooth irreducible variety and $\dim(TV) = 2d$. Moreover, if $I(V)=(f_1,\ldots,f_r)\subseteq K[\x]$ then $I(TV)=(f_1,\ldots,f_r,\nabla f_1\cdot\y,\ldots,\nabla f_r\cdot \y)\subseteq K[\x, \y]$, where $\y:=y_1,\ldots,y_n$ are new indeterminates.
\end{proposition}

\begin{proof}
Let $\pi_1:TV\to V$ be the projection on the first $n$ coordinates. Clearly $\pi_1(TV)=V$ and the fiber of any point $p\in V$ is the $d$-dimensional linear subspace $T_pV$, in particular it is irreducible.
Let $TV=W_1\cup \cdots \cup W_s$ be the decomposition of $TV$ into irreducible components and $W_k$  any of them. Let $(p,q)\in W_k\setminus \bigcup_{l\ne k} W_l$. Since the fiber $\pi_1^{-1}(p)$ is irreducible, it must be included in $W_k$. Considering now the dominant map $\pi_1:W_k\to \pi_1(W_k)\subseteq V$, from the Theorem on the dimension of fibers (see for instance \cite[Ch.I, \S6.3, Theorem 1.25]{shafarevich2013basic}), we conclude that $d=\dim (\pi_1^{-1}(p))\ge \dim(W_k)-\dim (\pi_1(W_k))\ge \dim(W_k)-d$ and so, $\dim(W_k)\le 2d$. On the other hand, since $V$ is smooth at $p$, it is locally an ideal theoretical complete intersection variety given by $n-d$ polynomials. Hence, $TV$ (and then also $W_k$) in a Zariski neighborhood of the point $(p,q)$ is defined locally by the common zeroes of $2(n-d)$ equations and, therefore, the local dimension of $W_k$ is at least $2n-2(n-d)=2d$. Summarizing we have shown that $\dim(W_k)=2d$. In other words, $TV$ is equidimensional and $\dim(TV)=2d$. 

For any point $(x,y)\in TV$ the Jacobian matrix of the polynomials $f_1({\x}),\ldots, f_r({\x}),$ $\nabla f_1({\x}) \cdot {\y},\ldots ,\nabla f_r({\x})\cdot{\y}$ has the block-form
\[
 \left(
\begin{matrix}
\dfrac{\partial (f_1,\ldots,f_r)}{\partial (x_1,\ldots,x_n)}(x)  & 0 \\
* & \dfrac{\partial (f_1,\ldots,f_r)}{\partial (x_1,\ldots,x_n)}(x)
\end{matrix}
\right)
\]
and so, its rank is at least $2(n-d)=2n-\dim TV$; then, from the Jacobian Criterion (see Theorem \ref{prop: jacobian}), $TV$ is smooth and $I(TV)=(f_1,\ldots,f_r,\nabla f_1\cdot\y,\ldots,\nabla f_r\cdot\y)$. In order to see that $TV$ is irreducible, it suffices to show that $TV$ is connected: suppose that $TV=C_1\cup\cdots\cup C_N$ is the decomposition in connected components. Thus, for any point $p\in V$ we have that $\pi_1^{-1}(p)$ is the disjoint union of $\pi_1^{-1}(p)\cap C_j$, for $j=1,\ldots,N$. Since $\pi_1^{-1}(p)$ is the linear subspace $T_pV$ we deduce that for each $p\in V$ there exists a unique connected component $C_j$ of $TV$ such that $p\in \pi_1(C_j)$. In other words, $V$ is the \emph{disjoint} union $\pi_1(C_1)\cup\cdots\cup \pi_1(C_N)$ and each $\pi_1(C_j)$ is a constructible set whose dimension is $d$ that is contained in $V$. Since $V$ is assumed irreducible, all the disjoint sets $\pi_1(C_j)$ contain a Zariski dense open subset of $V$. Hence $N=1$ and $TV$ is connected.
\end{proof}

\bigskip

Combining Propositions \ref{prop:mumford} and \ref{prop:dimTV} we deduce:

\begin{corollary}
Let $V\subseteq \A^n$ be an irreducible smooth variety. The ideal $I(TV)$ can be generated by polynomials of degree bounded by $\deg(V)$. $\square$
\end{corollary}

Our aim is to study the geometric degree of the tangent bundle of a smooth irreducible variety. 
We start noticing the following general lower bound:

\begin{proposition}\label{Prop: Lowerboundtrivial}
    Let $V \subseteq \mathbb{A}^n$ be a smooth irreducible algebraic variety. Then, $$\deg(TV) \geq \deg(V).$$
\end{proposition}
\begin{proof}
  Consider $\pi_1: \mathbb{A}^{2n} \to \mathbb{A}^{n}$ the projection to the first $n$ coordinates. Note that $\pi_1$ is linear and $\pi_1(TV) = V$; then, by \cite[Lemma 2]{Heintz1983}, we deduce that
$\deg(V) = \deg(\pi_1(TV)) \leq \deg(TV).$
\end{proof}

\bigskip
When the variety $V$ is a smooth hypersurface,  Bézout's inequality provides straightforwardly an upper bound for $\deg(TV)$. 

\begin{proposition}\label{proposition hipersup}
    Let $V \subseteq \mathbb{A}^n$ be a smooth irreducible hypersurface. Then:
    \begin{equation*}
        \deg(TV) \leq (\deg(V))^2.
    \end{equation*}
\end{proposition}
\begin{proof}
It suffices to note that, if $F$ is the reduced equation defining $V$, the tangent bundle $TV$ is the zero set of the polynomials $F$ and  $\nabla F \cdot \textbf{y}$ (see Proposition \ref{prop:dimTV}), both of degree bounded by $\deg(V)$; hence the result follows from Bezout's Theorem.
\end{proof}

\bigskip
In Section \ref{sec: upper_bounds} we will discuss the optimality of this bound.

\bigskip

To finish this short introduction concerning the tangent bundle we recall a classical functorial-type ingredient: the \emph{differential of a regular morphism} between smooth algebraic varieties.

Let $V\subseteq\A^n$ and $W\subseteq\A^m$ be affine smooth algebraic varieties and let $g:V\to W$ be a regular morphism. Suppose that $g=(g_1,\ldots,g_m)$ where $g_1,\ldots,g_m\in K[\x]$ and set $J$ for the Jacobian matrix $\dfrac{\partial(g_1,\dots,g_m)}{\partial (x_1,\ldots,x_n)}$. By the chain rule, the map $(x,y)\mapsto (g(x),J(x)\cdot y)$ defines a regular morphism between the algebraic varieties $TV$ and $TW$.
This morphism is called the \emph{differential} of $g$ and denoted by $Dg$. In fact, it is easy to see that $Dg$ depends only on the regular map $g$ but not on the choice of polynomials $g_1,\ldots,g_m$ defining it.

If $p\in V$, we write ${D_pg}$  for the linear map between $T_pV$ and $T_{g(p)}W$ given by the Jacobian matrix $J(p)$. As customarily we say that $g$ is an \emph{immersion} (resp. a \emph{submersion}) in $p\in V$ if this map is injective (resp. surjective).

\bigskip

In our analysis of the tangent bundle of a smooth irreducible variety $V\subseteq \A^n$ it will be useful to deal with an associated algebraic variety that we introduce below.

Let $\pi_1$ and $\pi_2$ be the canonical projections of $\A^n\times \A^n$ onto the first and the second coordinates respectively. As we observed before, $\pi_1(TV)=V$ and the fiber $\pi_1^{-1}(p)$ over each point $p\in V$ is exactly its $d$-dimensional tangent space $T_pV$.
On the other hand, $\pi_2(TV)$ is not necessarily closed but its closure is an irreducible variety (since $V$ is irreducible).

\begin{definition} \label{tangential}
The algebraic variety $\overline{\pi_2(TV)}\subseteq \A^n$ is called the \emph{tangential variety} of $V$ and is denoted by $\Tan(V)$.
\end{definition}

Roughly speaking,  $\textrm{Tan}(V)$ is (the closure of) the set of all the lines which are tangent at some point of $V$, i.e. $\textrm{Tan}(V)=\overline{\bigcup_{p\in V}(T_pV)}$. The dimension and regularity of $\textrm{Tan}(V)$ in general are not too easy to determine and depend strongly on the geometry of the variety $V$.
The fact that a linear map does not increase the geometric degree (see \cite[Lemma 2]{Heintz1983}) implies the following relation between the degrees of the tangential variety and the tangent bundle:

\begin{proposition} \label{prop:degTanVledegTV} Let $V\subseteq \A^n$ be a smooth irreducible algebraic variety. Then,
\[\deg(\Tan(V)) \le \deg(TV).\]
\end{proposition}

We finish this  subsection by proving a quite natural result concerning the linearity of $\textrm{Tan}(V)$ (see Proposition \ref{prop:tangential_lines} below). In order to prove it we need a known result relating generic projections of a variety and its degree (see for instance \cite[Proposition 45]{KPS2001} or \cite[Definition 18.1 (ii)]{harris1992algebraic}). We include the proof for the sake of completeness.

\begin{lemma} \label{Noether_grado}
Let $V\subseteq \A^n$ be a $d$-dimensional irreducible variety. There exists an affine linear epimorphism $h:\A^n\to\A^{d+1}$ such that $h$ restricted to $V$ is finite and $h(V)$ is a hypersurface of degree $\deg(V)$. Moreover, the property holds for a generic affine linear map $h:\A^n\to\A^{d+1}$.
\end{lemma}

\begin{proof}
From the genericity on the Noether Normalization Lemma and on the definition of degree, there exist generic polynomials of degree one $\ell_1,\ldots,\ell_{d}\in K[\x]$ algebraically independent such that if $R:=K[\ell_1,\ldots,\ell_{d}]$, the natural morphism $R\hookrightarrow K[V]$ is integral and injective and the set $V\cap\{\ell_1=0,\ldots,\ell_{d}=0\}$ has $\deg(V)$ many points.
Let $\xi$ be a new generic linear form (algebraically independent of the previous ones) such that $\xi$ takes different values in all these $\deg(V)$ many points. Let $m_\xi\in \textrm{Frac}(R)[T]$ be the minimal polynomial of the class of $\xi$ in $K[V]$ over the fraction field of $R$. Since $R$ is integrally closed, this polynomial defines also the integral dependence equation of the class of $\xi$ of minimal degree (see \cite[Proposition 5.15]{atiyah1994introduction}) and so $m_\xi\in R[T]$. Evaluating this polynomial in $\ell_1=\ell_2=\cdots=\ell_d=0$, we obtain a univariate polynomial which vanishes in the different values of $\xi$ over the points of $V\cap\{\ell_1=0,\ldots,\ell_{d}=0\}$. In particular $\deg_T(m_\xi)\ge \deg(V)$.

Geometrically, we have constructed an affine linear morphism $h:\A^n\to\A^{d+1}$, $h=(\ell_1,\ldots,\ell_d,\xi)$, such that $h:V\to h(V)$ is finite, $h(V)=\{m_{\xi}=0\}$ and $\deg(h(V))\ge \deg_T(m_\xi)\ge \deg(V)$. On the other hand, $\deg(h(V))$ is bounded by $\deg(V)$ because $h(V)$ is the image of an algebraic variety of degree $\deg(V)$ under an affine linear morphism (see for instance \cite[Lemma 2]{Heintz1983}).\end{proof}

\bigskip
Now we are able to prove:

\begin{proposition} \label{prop:tangential_lines}
Let $V\subseteq \A^n$ be a $d$-dimensional irreducible variety such that there exists a non-empty Zariski open set $\mathcal{U}\subseteq \operatorname{Reg}(V)$ such that for all $p,q\in \mathcal{U}$ the equality $T_pV=T_qV$ holds. Then $V$ is a linear variety.
\end{proposition}

\begin{proof}
We first prove the result for hypersurfaces and then, reduce the general case to this particular one.

If  $V$ is a hypersurface, it is given by a single equation $f=0$, with $f\in K[\x]$ irreducible. After a suitable linear change of coordinates we may suppose without loss of generality that, for any point $p\in\mathcal{U}$, we have $T_p(V)=\A^{n-1}\times \{0\}$. Therefore, all the derivatives $\partial f/\partial x_i$ vanish over $\mathcal{U}$ (hence over $V$) for $i=1,\ldots,n-1$. So, each of these partial derivatives which is not the zero polynomial must be divisible by $f$, but this is impossible by degree reasons. Then, $f$ is a polynomial in the single variable $x_n$ which is linear because it is irreducible and, therefore, $V$ is a hyperplane.

Suppose now that $V$ is not a hypersurface. Let $h:\A^n\to \A^{d+1}$ be as in Lemma \ref{Noether_grado} and set $W$ for the hypersurface $h(V)$. We know that $\deg(V)=\deg(W)$ and $\mathcal{G}:=h(\mathcal{U})\cap \textrm{Reg}(W)$ is a dense Zariski open subset of $W$. For any $p\in\mathcal{U}$ set $S:=T_pV$ (which does not depend on $p$). From the genericity of $h$ we may suppose also that $\ker(h)\cap S=0$, because $\dim(\ker(h))=n-(d+1)$ and $\dim(S)=d$.  Therefore, for any $p\in h^{-1}(\mathcal{G})\cap \mathcal{U}$, the differential $D_ph$ is surjective and independent of the point $p$, since $h$ is $K$-linear and $\ker(h)\cap T_p V=0$ for all  $p\in \mathcal{U}$. We conclude that $T_{h(p)}W=T_{h(q)}W$ for all $p,q\in h^{-1}(\mathcal{G})\cap \mathcal{U}$; in other words, $T_{p'}W=T_{q'}W$ for all $p',q'\in \mathcal{G}$. Now, by applying the result to the hypersurface $W$ and its open subset $\mathcal{G}$, we infer that $W$ is a linear variety. Therefore, $1=\deg(W)=\deg(V)$ and so, $V$ is a linear variety too (see for instance \cite[Chapter I, Exercise 7.6]{hartshorne1977algebraic}). 
\end{proof}
 
\section{The degree of the tangent bundle of a smooth curve} \label{sec:curves}

This section is devoted to studying the degree of the tangent bundle of a smooth irreducible algebraic curve. 
We will first prove a general result that expresses this degree in terms of intrinsic parameters associated to the curve. Then, we will focus on the particular case of curves given by rational or parametric representations.

We remark that, for the case of an arbitrary smooth \emph{plane} curve $\CC$, we have the optimal bound $\deg(T\CC)\le (\deg(\CC))^ 2$ (see Proposition \ref{proposition hipersup} for the upper bound and Example \ref{ejemplo optimalidad} below for the optimality).

\subsection{A formula for the degree}

We will start proving a formula for the degree of the tangent bundle of a smooth curve in terms of the degree of the curve, the degree of the tangential variety of the curve and an additional intrinsic invariant which we will introduce.

First, we determine the dimension of the tangential variety of a smooth curve (see also \cite[Exercise 8.1.2]{landsberg}).

\begin{lemma}\label{lem:tanC}
   If $\CC \subseteq  \mathbb{A}^n$ is a smooth irreducible algebraic curve which is not a line,  then $\dim(\operatorname{Tan}(\CC)) = 2$.
\end{lemma}
\begin{proof}
 Since $T\CC$ is irreducible and $\dim(T\CC)=2$ (see Proposition \ref{prop:dimTV}), we have that $\operatorname{Tan}(\CC)$ is also irreducible and $\dim(\operatorname{Tan}(\CC)) \leq 2$.
 Taking into account that $\operatorname{Tan}(\CC)$ contains a union of $1$-dimensional linear vector spaces, it is an infinite set. In addition, since $\CC$ is not a line, by Proposition \ref{prop:tangential_lines}, it has at least two different tangent directions and so, $\operatorname{Tan}(\CC)$ is an irreducible algebraic variety that contains at least two different $1$-dimensional linear spaces. Therefore, $\dim(\operatorname{Tan}(\CC))>1$. The result follows.
 \end{proof}

\bigskip

Due to Lemma \ref{lem:tanC}, for a smooth irreducible curve $\CC$ which is not a line, the projection $\pi_2:T\CC \rightarrow \operatorname{Tan}(\CC)$ is a dominant map between irreducible algebraic varieties of the same dimension. Then, Corollary \ref{coro:fibra} implies that there is a Zariski dense open subset of the image of $\pi_2$ such that all its points have finite fibers of the same cardinality. This motivates the following definition.

\begin{definition}
    Let $\CC \subseteq  \mathbb{A}^n$ be a smooth irreducible algebraic curve that is not a line. We define
    \begin{equation*}
        \omega(\CC) = [K(\operatorname{Tan}(\CC)):K(T\CC)] =\# \pi_2^{-1}(v), \text{ for a generic } v \in \operatorname{Tan}(\CC).
    \end{equation*}
For a line $\CC$, we define $\omega(\CC) = 0$.
\end{definition}

Before relating this invariant with the degree of the tangent bundle, we will prove an upper bound for $\omega(\CC)$ in terms of the degree of $\CC$.

\begin{proposition}\label{prop:wC-bound}
If $\CC \subseteq  \mathbb{A}^n$ is a smooth irreducible algebraic curve which is not a line, then
$\omega(\CC) \leq \deg(\CC)(\deg(\CC)-1)$.
\end{proposition}

\begin{proof}
     For a generic $v \in \operatorname{Tan}(\CC)$ such that $\omega(\CC)= \# \pi_2^{-1}(v)$, consider the set
\begin{equation*}
    \CC_{v} =
    \{ x \in \mathbb{A}^{n} \mid x \in \CC,\
          v \in T_{x}\, \CC \}.
\end{equation*}
Note that $\CC_{v}\subseteq  \CC$
is a finite set which is in one-to-one correspondence with $\pi_2^{-1}(v)$. Then, in order to get an upper bound for $\omega(\CC)$, it suffices to estimate the degree of $\CC_{v}$.

As $\CC$ is a smooth curve, according to Proposition \ref{prop:mumford}, its vanishing ideal can be generated by polynomials $f_1,\dots, f_r\in K[\x]$ of degrees bounded by $\deg(\CC)$. Therefore,
\[\CC_{v} = \CC \cap \bigcap_{1\le i \le r}\{x \in \mathbb{A}^n \mid \nabla f_i(x) \cdot v=0 \}.\]
Now, taking into account that each of the hypersurfaces in the intersection is defined by a polynomial of degree at most $\deg(\CC) - 1$, we apply the first inequality in Proposition \ref{prop:bezout}  to conclude that $\deg(\CC_{v}) \le \deg(\CC) (\deg(\CC)-1)$.
\end{proof}

\bigskip

Under the same assumptions as in Proposition \ref{prop:wC-bound}, 
it can be easily seen that, for every $v\in \Tan(\CC)$, $v\ne 0$, the fiber $\pi_2^{-1}(v)$ is either empty or consists of finitely many points.
If, in addition, the variety  $\operatorname{Tan}(\CC)$ is normal, we have that
$\# \pi_2^{-1}(v) \leq [K(\operatorname{Tan}(\CC)):K(T\CC)]$ for every $v \in \operatorname{Tan}(\CC)$, $v \ne 0$ 
(see Proposition \ref{prop:fibersNormal}),
and, as a consequence,
\begin{equation*}
    \omega(\CC) = \max_{v \in \operatorname{Tan}(\CC)\setminus\{0\}} \#\pi_2^{-1}(v) = \max_{p \in \CC} \#\{ q \in \CC \mid T_q \CC = T_p \CC \}.
\end{equation*}
Now, for an irreducible smooth \emph{plane} curve that is not a line, from Lemma \ref{lem:tanC}, it follows that its tangential variety is $\mathbb{A}^2$ (in particular, a normal variety). Therefore, we have:

\begin{corollary}
If $\CC \subseteq  \mathbb{A}^2$ is a smooth irreducible curve which is not a line,  then
\begin{equation*}
    \omega(\CC) = \max_{p \in \CC} \#\{ q \in \CC \mid T_p \CC = T_q \CC \}. \quad \square
\end{equation*}
\end{corollary}

The following example shows that the previous identity does not necessarily hold for an arbitrary smooth algebraic curve.

\begin{example}
Let $\CC \subseteq  \mathbb{A}^3$ be the algebraic curve parameterized by $\sigma(t) = \left (t, \frac{t^{3}}{3}-t,\frac{t^{4}}{4}-\frac{t^{2}}{2} \right )$.
For every $t$, we have that $T_{\sigma(t)}\CC$ is the linear space spanned by $\sigma'(t) = (1, t^2-1, t^3 - t)$. Then, $\operatorname{Tan}(\CC) = \{ (s, s(t^2 - 1), s(t^3-t)) \in \A^3 \mid s, t \in \C\}$.

For $v=(v_1,v_2,v_3) \in \operatorname{Tan}(\CC)$ with $v_2\ne 0$, the fiber $\pi_2^{-1}(v) $ consists of a single point $\sigma(t_0)$ with $t_0= {v_3}/{v_2}$;
therefore, $\omega(\CC)=1$. On the other hand, for $v\in \operatorname{Tan}(\CC)\setminus \{ 0\}$  with $v_2=0$,  we have that $\pi_2^{-1}(v) = \{(1,-\frac{2}{3}, -\frac{1}{4}), (-1,\frac{2}{3}, -\frac{1}{4})\}$.
We conclude that
$ \max_{p \in \CC} \#\{ q \in \CC \mid T_q \CC = T_p \CC \}=2.$
\end{example}

The main result of this section is the following:
\begin{theorem} \label{thm:degTC}
    Let $\CC \subseteq  \mathbb{A}^n$ be a smooth irreducible algebraic curve. Then,
    \begin{equation*}
        \deg(T\CC) = \deg(\CC) + \omega(\CC) \deg(\operatorname{Tan}(\CC)).
    \end{equation*}
\end{theorem}

In order to prove Theorem \ref{thm:degTC}, we will work  with varieties not only in affine spaces but also in projective and multi-projective spaces. We start recalling some definitions and a result from \cite{vdW1978} that will be key to our proof.

\bigskip

Let  $\mathbf{x}$ and $\mathbf{y}$ be two sets of variables. A polynomial in $K[\mathbf{x}, \mathbf{y}]$ is said to be \emph{bihomogeneous} in $(\mathbf{x}, \mathbf{y})$ if it is homogeneous both in the variables $\mathbf{x}$ and in the variables $\mathbf{y}$. A \emph{bihomogeneous ideal} of $K[\mathbf{x}, \mathbf{y}]$ is an ideal that can be generated by bihomogeneous polynomials in $(\mathbf{x},\mathbf{y})$.

\begin{lemma}\label{asd}
    Let $X \subseteq  \mathbb{A}^{n+m}$ be an algebraic variety defined by polynomials in $K[\x,\y]$ with $\x=(x_1,\dots, x_n)$ and  $\mathbf{y}=(y_1,\dots, y_m)$ such that, for every $(x,y) \in X$, we have that $(x,\lambda y) \in X$ for all $\lambda \in K$. Let $\overline{X} \subseteq  \mathbb{P}^{n+m}$ be the projective closure of $X$. Then, the ideal $I(\overline{X}) \subseteq  K[\overline{\mathbf{x}},\mathbf{y}]$ is a bihomogeneous ideal in $\overline{\mathbf{x}}= (x_0,x_1,\dots, x_n)$ and $\mathbf{y}$.
\end{lemma}

\begin{proof}
Let $f \in I(\overline{X})$ be a homogeneous polynomial of degree $D$. Then, we may write
\[f(\overline{\x}, \y)= \sum_{k=0}^D \Big(\sum_{|\alpha|=k} c_{\alpha}(\overline{\x}) \y^\alpha \Big)\]
where $c_{\alpha}(\overline{\x})\in K[\overline{\x}]$ is homogeneous of degree $D-k$ for every $\alpha=(\alpha_1,\dots, \alpha_m)\in (\Z_{\ge 0})^m$ with $|\alpha|=\sum_i \alpha_i = k$.

For a point $(x,y) \in X$, by the assumption on $X$,  we have that $(1: x: \lambda y)\in \overline{X}$ for all $\lambda \in K$ and so,
\[f(1,x, \lambda y)= \sum_{k=0}^D \Big(\sum_{|\alpha|=k} c_{\alpha}(1,x) y^\alpha \Big) \lambda^k\]
vanishes for all $\lambda \in K$. This implies that $\sum_{|\alpha|=k} c_{\alpha}(1,x) y^\alpha =0$ for $k=0,\dots, D$. We conclude that $\sum_{|\alpha|=k} c_{\alpha}(1,\x) \y^\alpha\in I(X)$ and, homogenizing up to degree $D$, we obtain that $f^{(k)}:=\sum_{|\alpha|=k} c_{\alpha}(\overline{\x}) \y^\alpha\in I(\overline{X})$,  for $k=0,\dots, D$.
Therefore,  $f= \sum_{k=0}^D f^{(k)}$ with $f^{(k)} \in I(\overline{X})$ bihomogeneous in $(\overline{\x},\y)$.

Thus, if $f_1,\dots, f_r\in I(\overline{X})$ are homogeneous polynomials that generate $I(\overline{X})$, the bihomogeneous polynomials $f_i^{(k)} \in K[\overline{\x}, \y]$ defined as before, for $i=i,\dots, r$, also generate $I(\overline{X})$.
\end{proof}

\bigskip

Let $X \subseteq  \mathbb{P}^{n+m+1}$ be an equidimensional variety defined by homogeneous polynomials in $K[\mathbf{x}, \mathbf{y}]$, where $\mathbf{x}$ and $\mathbf{y}$ are sets of $n+1$ and $m+1$ variables respectively. Assume that the polynomials defining $X$ are bihomogeneous in $(\mathbf{x}, \mathbf{y})$. Then, they define a variety $X'$ in $\mathbb{P}^n \times \mathbb{P}^m$. Note that, if $\dim(X) = d$, then $X'$ is equidimensional of dimension $d-1$.

We may associate to $X'\subseteq  \mathbb{P}^n \times \mathbb{P}^m$ a family of bidegrees: for $(a,b) \in (\Z_{\ge 0})^2$ such that $a+b= d-1$,
the intersection of $X'$ with a linear variety of the form \[\{ (x , y) \in \mathbb{P}^n \times \mathbb{P}^m \mid L^n_i(x) = 0, \ 1 \le i \le a; L^m_j(y) = 0, \ 1\le j \le b\},\] where $L_i^n(\x)$ and $L_j^m(\y)$ are generic linear forms in $n+1$ and $m+1$ variables respectively, is a finite set. The number of points in this set is called
the \emph{$(a,b)$-bidegree of $X'$}, and will be denoted by $\deg_{(a,b)}(X')$.
According to the main result in \cite{vdW1978}, the bidegrees of $X'$ are related to the degree of $X$ as follows:

\begin{equation} \label{eq:vdw}
        \deg(X) = \sum_{a+b=d-1} \deg_{(a,b)}(X'). 
\end{equation}

We go from the affine space $\A^n$ to the projective space $\mathbb{P}^{n-1}$ by means of the usual map $\psi:\A^n \setminus \{0\} \to \mathbb{P}^{n-1}$, $\psi(y_1,\dots, y_n) = (y_1: \dots: y_n)$. For an algebraic set $X\subseteq \mathbb{P}^{n-1}$, we may consider the \emph{affine cone over $X$} in $\A^n$ defined as $CX = \{ 0\}\cup\{ y\in \A^n\setminus\{0\} \mid \psi(y) \in X\} $. It is easy to see that if $X$ is an irreducible projective variety, then $\deg(X) = \deg(CX)$.

\bigskip
We are now ready to prove Theorem \ref{thm:degTC}.

\bigskip

\begin{proof}[Proof of Theorem \ref{thm:degTC} (Theorem A in the Introduction)]  First note that the formula holds trivially if $\CC$ is a line, since in this case $\omega(\CC) =0$ by definition.
Assume now that the curve $\CC$ is not a line.

Consider $T\CC\subseteq \A^{2n}$ and its projective closure $\overline{T\CC}\subseteq \mathbb{P}^{2n}$. Note that $\overline{T\CC}$ is an irreducible projective variety of dimension $2$ and, by Lemma \ref{asd},  its defining ideal $I(\overline{T\CC})\subseteq  K[\overline{\x},\y]$ is bihomogeneous. Hence, applying Identity \eqref{eq:vdw}, we have that
\begin{equation*}
    \deg(T\CC) = \deg(\overline{T\CC}) = \deg_{(1,0)}(\overline{T\CC}') + \deg_{(0,1)}(\overline{T\CC}').
\end{equation*}
To prove the identity in our statement, we will show that
\[\deg_{(1,0)}(\overline{T\CC}') = \deg(\CC) \ \hbox{ and } \ \deg_{(0,1)}(\overline{T\CC}') = \omega(\CC) \deg(\operatorname{Tan}(\CC)).\]
As we will see, the computation of these bidegrees amounts to computing the degree of the projection of $\overline{T\CC}'\subseteq  \mathbb{P}^n \times \mathbb{P}^{n-1}$ to $\mathbb{P}^n$ and $\mathbb{P}^{n-1}$ respectively, and adding up the cardinalities of all the fibers over each of the finite set of points in the corresponding projection.

By definition, the first bidegree equals the number of points in the intersection
\begin{equation}\label{eq:deg10}
    \overline{T\CC}' \cap \{ (\overline{x},y) \in \mathbb{P}^{n} \times \mathbb{P}^{n-1} \mid L_1(\overline{x}) = 0 \}
\end{equation}
for a generic linear form $L_1(\overline{\x})$. Since $\overline{T\CC}' \cap \{x_0 = 0\} = (\overline{T\CC} \cap \{x_0 = 0\})'$ and the variety $\overline{T\CC} \cap \{x_0 = 0\}\subseteq \mathbb{P}^{2n}$ is equidimensional of dimension $1$, we have that $\overline{T\CC}' \cap \{x_0 = 0\}\subseteq \mathbb{P}^n \times \mathbb{P}^{n-1}$ is a finite set of points. We may assume that the generic linear form $L_1$ does not vanish at any of these points and then, the intersection \eqref{eq:deg10} is contained in the open set $\{ x_0 \ne 0 \}\subseteq \mathbb{P}^n \times \mathbb{P}^{n-1}$.

For a point $(\overline{p}, v) \in \overline{T\CC}' \cap \{x_0 \ne 0\}$ we can take coordinates of the form $((1: p), v)$. It follows that $(1:p:v)\in \overline{T\CC} \cap \{ x_0\ne 0\} $ and so, $(p,v)\in T\CC$. Note that the fact that $T\CC$ is defined by equations that are homogeneous in the variables $\y$ implies that this will happen independently of the coordinates we choose to represent $v\in \mathbb{P}^{n-1}$. Conversely, for every point $(p,v) \in T\CC$ with $v\ne 0$, we have that $(1:p:\lambda v)\in \mathbb{P}^{2n}$ satisfies the defining equations of $\overline{T\CC}$ for every $\lambda \in K$ and, therefore, $((1:p), v) \in \mathbb{P}^n \times \mathbb{P}^{n-1}$ lies in $\overline{T\CC}' \cap \{x_0 \ne 0\}$.
Therefore, if $\ell_1(\x):=L_1(1,\x)$ and $f_1,\dots, f_r \in K[\x]$ are generators of the ideal $I(\CC)$, under the previous genericity assumption, it follows that the intersection \eqref{eq:deg10} equals
\[\Bigg\{((1:x),y) \in \mathbb{P}^n \times \mathbb{P}^{n-1} \mid \left\{\begin{array}{l} f_1(x) = 0, \dots, f_r(x) = 0, \ell_1(x)=0, \\ \nabla f_1(x) \cdot y = 0, \dots, \nabla f_r(x) \cdot y =0 \end{array}\right.\Bigg\}.\]

The system $f_1(x) = 0, \dots, f_r(x) = 0, \ell_1(x)=0$ defines $\CC \cap \{ \ell_1(x)=0\}$; then, it has $\deg(\CC)$ solutions for a generic $\ell_1$. In addition, for each solution $p$ to this system, the equations $\nabla f_1(p) \cdot y = 0, \dots, \nabla f_r(p) \cdot y =0$ define $T_{p}\CC$ and so, since $\CC$ is smooth, they have a unique solution in $\mathbb{P}^{n-1}$. It follows that $\deg_{(1,0)}(\overline{T\CC}') = \deg(\CC)$.

Now, to prove the formula for $\deg_{(0,1)}(\overline{T\CC}')$, we will count the number of points in the intersection
\begin{equation*}
    \overline{T\CC}' \cap \{ (\overline{x},y) \in \mathbb{P}^{n} \times \mathbb{P}^{n-1} \mid L_2(y) = 0 \}
\end{equation*}
for a generic linear form $L_2(\y)$. Let $\Pi_2: \overline{T\CC}' \to \mathbb{P}^{n-1}$ be the projection to the second factor, $\Pi_2(\overline{x}, y) = y$. Then,
\begin{equation}\label{eq:union}
\overline{T\CC}' \cap \{ (\overline{x},y) \in \mathbb{P}^{n} \times \mathbb{P}^{n-1} \mid L_2(y) = 0 \} = \bigcup_{v\in \Pi_2(\overline{T\CC}') \cap \{L_2(y)=0\}} \Pi_2^{-1}(v).
\end{equation}

We start establishing a relation between $\Pi_2(\overline{T\CC}')$ and $\Tan(\CC)$.

Consider the map $\psi:\A^n\setminus \{0\}\to \mathbb{P}^{n-1}$,  $\psi(y_1, \dots, y_n) = (y_1:\dots: y_n)$, and let $\PTan(\CC):= \psi(\Tan(\CC) \setminus \{0\})$. Note that $\pi_2(T\CC)$ satisfies that, if $v\in\pi_2(T\CC)$ then $\lambda v\in \pi_2(T\CC)$ for every $\lambda \in K$; then, $I(\Tan(\CC))= I(\pi_2(T\CC))$ is a homogeneous ideal in $K[\overline{\y}]$ and so, $\PTan(\CC)$ is the projective variety defined in $\mathbb{P}^{n-1}$ by this ideal. In particular, we have that $\dim(\PTan(\CC)) =1$ and $\Tan(\CC)$ is the affine cone over $\PTan(\CC)$.

We claim that $\Pi_2(\overline{T\CC}')= \PTan(\CC)$. To prove this equality, let us observe that
\begin{equation}\label{eq:projpi2}
\psi(\pi_2(T\CC)\setminus \{0\})
= \{ y\in \mathbb{P}^{n-1} \mid \exists x \in K^n: f(x,y)= 0 \ \forall f\in I(T\CC)\}
\end{equation}
and, recalling that $\overline{T\CC}'$ is defined in $\mathbb{P}^n \times \mathbb{P}^{n-1}$ by the same polynomials as the projective closure $\overline{T\CC}$ in $\mathbb{P}^{2n}$,
\begin{equation}\label{eq:Pi2}
\Pi_2(\overline{T\CC}') 
= \{ y\in \mathbb{P}^{n-1} \mid \exists (x_0,x) \in K^{n+1}\setminus \{ 0\} : f^h(x_0,x,y)= 0 \ \forall f\in I(T\CC)\}.
\end{equation}
Here, we write $f^h\in K[\overline{\x},\y]$ for the homogenization with  $x_0$ of a polynomial $f\in K[\x, \y]$.

It is straightforward to see from \eqref{eq:projpi2} and \eqref{eq:Pi2} that $\psi(\pi_2(T\CC)\setminus \{0\})\subseteq  \Pi_2(\overline{T\CC}')$.
On the other hand, we have that $\Pi_2(\overline{T\CC}' \cap \{ x_0 \ne 0\})$ is a dense open subset of the closed set $\Pi_2(\overline{T\CC}')$ (since $\overline{T\CC}'$ is not contained in the hyperplane $\{ x_0 = 0\}$ in $\mathbb{P}^n \times \mathbb{P}^{n-1}$) which is included in $\psi(\pi_2(T\CC)\setminus \{0\})$. 
It follows that $\Pi_2(\overline{T\CC}') = \overline{\psi(\pi_2(T\CC)\setminus \{0\})} = \PTan(\CC)$.

Therefore, for a generic linear form $L_2\in K[\y]$,
\[\#  (\Pi_2(\overline{T\CC}') \cap \{L_2(y)=0\}) = \deg (\PTan(\CC)) = \deg(\Tan(\CC)), \]
where the last equality holds due to the fact that $\Tan(\CC)$ is the affine cone over $\PTan(\CC)$.

To finish the proof, we will show that, for a generic linear form $L_2(\y)$, for every $v\in \Pi_2(\overline{T\CC}') \cap \{L_2(y)=0\}$, the cardinality of $\Pi_2^{-1}(v)$ equals $\omega(\CC)$.

Let $\mathcal{U}\subset\A^n$ be an open subset of $\Tan(\CC)$ such that $\# \pi_2^{-1}(v) = \omega(\CC)$ for all $v\in \mathcal{U}$. In particular, $0\notin \mathcal{U}$ and, for every $v\in \mathcal{U}$, $\pi_2^{-1}(v) = \pi_2^{-1}(\lambda v)$ for all $\lambda \in K\setminus\{0\}$. Since $\Tan(\CC)$ is defined by homogeneous polynomials, we have that $\lambda \mathcal{U} :=\{ \lambda y: y\in \mathcal{U}\}$ is an open subset of $\Tan(\CC)$ for every $\lambda \in  K\setminus\{0\}$. Then, $\widetilde{\mathcal{U}}:= \bigcup_{\lambda \in K\setminus\{0\}} \lambda \mathcal{U}$ is an open subset of $\Tan(\CC)$ satisfying the same condition as $\mathcal{U}$ on the fibers of $\pi_2$ and, furthermore, $\psi(\widetilde{\mathcal{U}})$ is a dense open subset of $\PTan(\CC)$.
Let $\mathcal{V}:= \Pi_2(\overline{T\CC}') \setminus \Pi_2(\overline{T\CC}' \cap \{x_0=0\})$, which is also a dense open subset of $\Pi_2(\overline{T\CC}')=\PTan(\CC)$, since $\overline{T\CC}' \cap \{x_0=0\}$ is a finite set. Finally, consider $\mathcal{G}:=  \psi(\widetilde{\mathcal{U}}) \cap \mathcal{V}$.

For generic coefficients, a linear form $L_2\in K[\y]$ satisfies $\#(\PTan(\CC) \cap \{ L_2(y)=0\}) = \deg(\PTan(\CC))$ and $(\PTan(\CC)\setminus \mathcal{G}) \cap \{L_2(y)=0\} = \emptyset$.
Therefore, $\Pi_2(\overline{T\CC}') \cap \{L_2(y)=0\}$ consists of $\deg(\Tan(\CC))$ points contained in $\mathcal{G}$.  Now, for  $v\in \mathcal{G}$, we have that $\# \pi_2^{-1}(v) = \omega(\CC)$. In addition,
$\Pi_2^{-1}(v) \subseteq \{ x_0 \ne 0\}$ and, as a consequence, for each of its points $(\overline{p},v)$ we may assume $\overline{p}=(1:p)$. Hence, $(p,v) \in T\CC$ and so, $(p,v) \in \pi_2^{-1}(y)$. Also, it is straightforward to see that a point $(p,v) \in \pi_2^{-1}(v)$ leads to a point $(\overline{p}, v) \in \Pi_2^{-1}(v)$. Therefore, $\# \Pi_2^{-1}(v) = \# \pi_2^{-1}(v) = \omega(\CC)$.

The equality $\deg_{(0,1)}(\overline{T\CC}') = \omega(\CC) \deg(\operatorname{Tan}(\CC))$ follows from \eqref{eq:union}.
\end{proof}

\subsection{Parametric algebraic curves}\label{subsection: parametric algebraic curves}

In this subsection, we will deal with the case of smooth curves given by polynomial or rational parametric representations. We will show that, under certain assumptions on the parametrization, the tangent bundle of the curve also has a parametric representation, and we will relate the degrees of the curve and its tangent bundle by means of the degrees of the polynomials involved in their parametric representations.

Let $\CC\subseteq \A^n$ be an algebraic curve.
We say that $\CC$ has a \emph{polynomial} (resp.~\emph{rational}) \emph{parametric representation} if there exist  polynomials $g_1,\dots, g_n \in K[t]$ (resp.~rational functions $g_1/h_1, \dots, g_n/h_n\in K(t)$) defining a dominant map $\mathcal{P}\colon \A^1 \to \CC$. Note that a curve with a polynomial or rational parametric representation is irreducible.

We will focus on smooth curves possessing a parametric representation $\mathcal{P}$ that satisfies, additionally:
\begin{itemize}
\item[(P1)] there is a non-empty open subset $\mathcal{U}\subseteq \A^1$ such that $\mathcal{P}:\mathcal{U} \to \CC$ is injective (that is, $\mathcal{P}$ is a \emph{proper} parametrization);
\item[(P2)] for every $t\in \mathcal{U}$, the derivative $\mathcal{P}'(t)$ does not vanish.
\end{itemize}

From a parametric representation $\mathcal{P}(t) = (g_1(t)/h_1(t), \dots, g_n(t)/h_n(t))$ of an algebraic curve $\CC$, we can obtain a parametric representation with a common denominator. Assume $\gcd(g_i, h_i) = 1$ for $i=1,\dots, n$. Let $g_0\in K[t]$ be the least common multiple of  $h_1,\dots, h_n$ and, for $i=1,\dots, n$, set $\widetilde{g}_i := g_i g_0/h_i\in K[t]$. Then, $\widetilde{\mathcal{P}}(t)= (\widetilde{g}_1(t)/g_0(t), \dots, \widetilde{g}_n(t)/g_0(t))$ is a parametric representation of $\CC$ with $\gcd(g_0, \widetilde{g}_1,\dots, \widetilde{g}_n)=1$. Moreover, if $\mathcal{P}$ satisfies (P1) (resp.~(P2)), then $\widetilde{\mathcal{P}}(t)$ also does.

\bigskip

We start proving a basic result on the geometric degree of curves given by parametric representations.
Upper bounds for the degree of parameterized varieties of arbitrary dimension can be found, for instance, in \cite{DHM2022} and the references therein.

\begin{proposition}\label{prop:degCparam}
Let $\CC\subseteq \A^n$ be an algebraic curve with a parametric representation $\mathcal{P}$ satisfying condition (P1). If $\mathcal{P}(t) = (g_1(t)/g_0(t),\dots, g_n(t)/g_0(t))$ with $g_0,g_1,\dots, g_n \in K[t]$ and $\gcd(g_0,g_1, \dots, g_n) =1$, then
$\deg(\CC) = \max\{\deg(g_i); 0\le i \le n\}$.
\end{proposition}

\begin{proof} To compute the degree of $\CC$, we consider its intersection with a hyperplane $H_a:=\{ x\in \A^n\mid a_0+a_1 x_1+\cdots +a_n x_n=0\}$ defined from $a=(a_0,a_1,\dots, a_n)\in\A^{n+1}$.

Assume $\mathcal{P} :\mathcal{U} \to \CC$ is injective for a non-empty Zariski open set $\mathcal{U} \subseteq \A^1$. Since the image of $\mathcal{P}$ contains a Zariski dense open subset of $\CC$, for generic $a\in \A^{n+1}$, the intersection $\CC \cap H_a$ is included in the image of $\mathcal{P}$. Then,
\[\CC \cap H_a =\{ x\in \A^n\mid \exists t\in \mathcal{U} : x= \mathcal{P}(t) \hbox { and } a_0 g_0(t)+a_1 g_1(t) +\cdots+ a_ng_n(t) =0\}.\] The injectivity of $\mathcal{P}$ implies that this set is in one to one correspondence with the set of roots in $\mathcal{U}$ of the polynomial $G_a:=a_0 g_0+a_1 g_1 +\cdots+ a_ng_n \in K[t]$.
Note that, for generic $a\in \A^{n+1}$, all the roots of $G_a$ lie in $\mathcal{U}$.

For $a\in \A^{n+1}$ with $G_a\ne 0$, we have that $\deg(G_a) \le \delta(\mathcal{P}):= \max\{\deg(g_i); 0\le i \le n\}$. Moreover, if $I:=\{ 0\le i \le n \mid \deg(g_i) = \delta(\mathcal{P})\}$, the equality holds for every $a\in \A^{n+1}$ such that $\sum_{i\in I}  \text{lc}(g_i)\cdot a_i \ne 0$. Consider the resultant $R(a):= \text{Res}_t(G_a, G'_a)\in K[a]$, regarding $G_a$ as a polynomial in the variable $t$ of degree $\delta(\mathcal{P})$. To finish the proof, it suffices to show that $R(a)$ is not the zero polynomial, since this implies that, for generic $a\in \A^{n+1}$, the polynomial $G_a$ has $\delta(\mathcal{P})$ distinct roots.

Since the degree of $G_a$ in the variables $a=(a_0,\dots, a_n)$ is $1$ and its coefficients $g_0(t), \dots, g_n(t)$ (regarded as a polynomial in $K[t][a]$) are relatively prime by our assumptions, it follows that $G_a$ is irreducible in $K[t,a]$. In addition, as $\deg_t(G'_a)< \deg_t(G_a)$, then $G_a$ does not divide $G'_a$ and so, $G_a$ and $G'_a$ do not have a common factor in $K[t,a]$ of positive degree in $t$. We conclude that the resultant $\text{Res}_t(G_a, G'_a)\in K[a]$ is not zero.
\end{proof}

\bigskip

We will now analyze the tangent bundle of a smooth curve with a parametric representation.
In a similar way as for curves, we can define the notion of polynomial or rational parametric representation for surfaces and higher dimensional irreducible varieties. Here, we will work with parametric representations of the tangent bundle of a curve.

\begin{lemma}\label{lem:paramTC}
Let $\CC\subseteq \A^n$ be a smooth irreducible curve with a (polynomial or rational) parametric representation $\CP: \A^1\to \CC$ satisfying the conditions (P1) and (P2). Then, $T\CC$ has a parametric representation given by $\widehat{\CP}: \A^2 \to T\CC$, $\widehat{\CP}(t,s) = (\CP(t), s \CP'(t))$.
\end{lemma}

\begin{proof} It follows straightforwardly noticing that $\widehat{\CP}$ is the differential of $\CP$.
\end{proof}

\bigskip

For curves having a polynomial parametric representation we have an exact formula for the degree of the tangent bundle in terms of the degree of the curve.

\begin{theorem}\label{thm:degTCpolyparam}
Let $\CC\subseteq \A^n$ be a smooth algebraic curve having a polynomial parametric representation that satisfies conditions (P1) and (P2). Then $\deg(T\CC) = 2 \deg(\CC) - 1$.
\end{theorem}

\begin{proof}
Let $\CP(t) =(g_1(t),\dots, g_n(t))$ be a polynomial parametric representation of $\CC$ satisfying (P1) and (P2).
Under our assumptions, by Lemma \ref{lem:paramTC}, the tangent bundle $T\CC\subseteq \A^{2n}$ is a surface with a polynomial parametric representation $\widehat{\CP}:\A^2 \to T\CC$ given by $\widehat{\CP}(t,s) = (g_1(t),\dots, g_n(t), s g_1'(t),\dots, s g'_n(t))$.
Note that, if $\mathcal{U}\subseteq \A^1$ is an open set such that $\CP$ satisfies (P1)  and (P2) in $\mathcal{U}$, then $\widehat{\CP}$ is injective in $\widehat{\mathcal{U}}:=\mathcal{U} \times \A^1$. Without loss of generality, we may assume $\mathcal{U} \subseteq \A^1\setminus \{0\}$.

To compute the degree of $T\CC$, we intersect it with linear varieties of codimension $2$:
\begin{equation}\label{eq:plane}
\Pi_{(a,b)}=\Big\{(x,y) \in \A^{2n} \mid a_{10}+\sum_{1\le i\le n}  a_{1i}x_i + \sum_{1\le j \le n} b_{1j} y_j =0, a_{20}+\sum_{1\le i\le n}  a_{2i}x_i + \sum_{1\le j \le n} b_{2j} y_j  =0\Big\}
\end{equation}
for $(a,b)\in \A^{4n+2}$ such that the rank of the associated coefficient matrix equals 2, and estimate the number of points in the intersection. Note that, for generic $(a,b)\in \A^{4n+2}$, we have $\Pi_{(a,b)} \cap T\CC\subseteq \A^n \times (\A^n\setminus \{0\})$, because $T\CC\cap (\A^n\times \{0\}) = \CC \times \{ 0\}$ has dimension $1$.
We will now use the parametric representation of $T\CC$ to reduce the problem to counting the number of solutions to a polynomial system.

For generic $(a,b)\in \A^{4n+2}$, the intersection $\Pi_{(a,b)} \cap T\CC$ is contained in the image of $\widehat{\CP}: \widehat{\mathcal{U}}\to T\CC$, since this image contains an open subset of $T\CC$. This implies that $\Pi_{(a,b)}\cap \CC$ is in one to one correspondence with
\[\Big\{(t,s) \in \widehat{\mathcal{U}} \mid a_{10}+\sum_{1\le i\le n}  a_{1i}g_i(t) + s \sum_{1\le j \le n} b_{1j} g_j'(t) =0, a_{20}+\sum_{1\le i\le n}  a_{2i}g_i(t) + s \sum_{1\le j \le n} b_{2j} g_j'(t)  =0\Big\} \]
For $k=1,2$, let $G_k(t,s) := a_{k0}+\sum_{1\le i\le n}  a_{ki}g_i(t) + s \sum_{1\le j \le n} b_{kj} g_j'(t) $. By our previous considerations, the degree of $T\CC$ equals the number of solutions in $\widehat{\mathcal{U}}$ of the system
\begin{equation}\label{eq:systemdeg}
G_{1}(t,s)=0, \ G_{2}(t,s)=0
\end{equation}
for generic coefficients $(a,b)\in \A^{4n+2}$.

Note that, for a fixed $t_0 \in K$, the polynomials $G_{1}(t_0,s)$ and  $G_{2}(t_0,s)$ do not have a common root in $K$ for generic $(a,b) \in \A^{4n+2}$, since they are two linear equations in $s$ with linearly independent coefficient vectors. Therefore, for generic $(a,b)$, all the solutions to \eqref{eq:systemdeg} in $(K\setminus\{0\})^2$ lie in $\widehat{\mathcal{U}}$.

Finally, we estimate the number of solutions of the system \eqref{eq:systemdeg} in $(K\setminus\{0\})^2$ by means of Bernstein's Theorem (see Proposition \ref{prop:BKK}).

To this end, let us determine the Newton polytopes of the polynomials $G_k$. These polynomials are of the form $G_k(t,s) = G_{k0}(t) + s G_{k1}(t)$, where $\deg(G_{k0}) \le \delta(\CP)$ and $\deg(G_{k1})\le \delta(\CP)-1$ with equalities for generic $(a,b)$. Then,
their Newton polytopes are contained in the polygon $Q$ with vertex set $\{(0,0), (\delta(\CP), 0), (\delta(\CP)-1, 1), (0,1)\}$. Moreover, the points $(\delta(\CP), 0)$ and $(\delta(\CP)-1, 1)$ are exponent vectors of monomials with non-zero coefficients in $G_k$ for generic $(a,b)$, because of the previous degree estimates, and $(0,0)$ corresponds to the term $a_{k0}$. The  monomial $s$, whose exponent vector is $(0,1)$, also appears with non-zero coefficient in $G_k$ for generic $(a,b)$ since, by condition (P2), $(g_1'(0),\dots, g'_n(0)) \ne 0$ and so, $G_{k1}(0) \ne 0$.
We conclude that, for generic $(a,b)$, the Newton polytope of $G_k$ is $Q$.

Therefore, Bernstein's Theorem (see Proposition \ref{prop:BKK}) states that the number of common roots of $G_1$ and $G_2$ in $(K\setminus\{0\})^2$ is at most $2 \,\hbox{area}(Q) = 2 \delta(\CP) - 1$. Furthermore, it ensures that the bound is attained provided that the associated polynomial systems supported on the faces of $Q$ do not have common zeros.

The systems associated with the faces determined by $(0,0)$ and $(\delta(\CP), 0)$ and by $(0,0)$ and $(0,1)$ are $\{G_1(t,0)=0, G_2(t,0)=0\}$ and $\{G_1(0,s)=0, G_2(0,s)=0\}$ respectively, and we have already seen that they do not have common zeros. For the face determined by $(0,1)$ and $(0, \delta(\CP)-1)$ the associated system is $\{sG_{11}(t) = 0, sG_{21}(t) = 0\}$, which does not have solutions in $(K\setminus\{0\})^2$ since $G_{11}(t)$ and $G_{21}(t)$ are generic linear combinations of the relatively prime polynomials $g'_1,\dots, g'_n$ and so, they do not have common roots. Finally, the face determined by $(\delta(\CP),0)$ and $(\delta(\CP)-1, 1)$ leads to the system $\{ t^{\delta(\CP) - 1} (\ell_{11}(a) t +  \ell_{12}(b)s ) = 0,  t^{\delta(\CP) - 1} (\ell_{21}(a) t+  \ell_{22}(b)s ) = 0 \}$, where $\ell_{k1}(a) = \sum_{i\in I} \hbox{lc}(g_i) a_{ki}$ and $\ell_{k2}(b) = \sum_{i\in I} \deg(g_i) \hbox{lc}(g_i) b_{ki}$  with $I=\{ 1\le i \le n\mid \deg(g_i) = \delta(\CP)\}$, which does not have solutions in $(K\setminus\{0\})^2$ for generic $a,b$.

From Proposition \ref{prop:degCparam}, we now conclude that $\deg(T\CC) = 2\delta(\CP) - 1  = 2 \deg(\CC) - 1$.
\end{proof}

\bigskip

Finally, for curves with a rational parametric representation we can prove an upper bound for $\deg(T\CC)$ in terms of $\deg(\CC)$. 

\begin{theorem}\label{thm:degTCratparam}
Let $\CC\subseteq \A^n$ be a smooth algebraic curve having a rational parametric representation that satisfies conditions (P1) and (P2). Then $\deg(T\CC) \le 3 \deg(\CC) - 2$.
\end{theorem}

\begin{proof}
Let $\mathcal{P}:\mathcal{U}\to \CC$, $\mathcal{P}(t) = (g_1(t)/g_0(t),\dots, g_n(t)/g_0(t))$, with $g_0,g_1\dots,g_n\in K[t]$ relatively prime polynomials and $\mathcal{U}\subseteq \A^1$ a Zariski dense open set, be a parametric representation of $\CC$ satisfying conditions (P1) and (P2). 

Without loss of generality,  we assume $\deg(g_0) > \deg(g_i)$ for $i=1,\dots, n$: if this is not the case, we can make the following reparametrizations and translations that do not modify the degrees of $\CC$ or $T\CC$, and preserve (P1) and (P2). First, by changing $t \mapsto t+c$ for a suitable $c\in K$, we may assume that $g_i(0) \ne 0$ for $i=0,\dots, n$. Then, by changing $t\mapsto t^{-1}$ and multiplying numerators and denominators by suitable powers of $t$, we get a rational parametric representation $(h_1(t)/h_0(t), \dots, h_n(t)/h_0(t))$ of $\CC$ where all the polynomials involved have the same degree. Finally, by Euclidean division, we write $h_i = q_i h_0 + r_i$ with $q_i \in K$ and  $\deg(r_i) <\deg(h_0)$ for $i=1,\dots, n$, and  consider the parametric representation $(r_1(t)/h_0(t), \dots, r_n(t)/h_0(t))$ of the translated curve $\CC + (-q)$, where $q=(q_1,\dots, q_n)$.

Then, by Proposition \ref{prop:degCparam}, we have $\deg(\CC) = \max \{\deg(g_i);0\le i \le n\} = \deg(g_0)$.

Now we proceed similarly as in the polynomial case. Applying Lemma \ref{lem:paramTC}, we obtain an injective parametric representation $\widehat{\mathcal{P}}: \widehat{\mathcal{U}} \to T\CC$, with $\widehat{\mathcal{U}}= \mathcal{U} \times \A^1 \subseteq \A^2$:
\[\widehat{\mathcal{P}}(t,s) = \Big(\dfrac{g_1(t)}{g_0(t)}, \dots, \dfrac{g_n(t)}{g_0(t)}, s \dfrac{g_1'(t)g_0(t)-g_1(t)g_0'(t)}{g_0^2(t)},\dots, s \dfrac{g_n'(t)g_0(t)-g_n(t)g_0'(t)}{g_0^2(t)}\Big)\]
We use this parametric representation to count the number of points in the intersection of $T\CC$ with a generic plane $\Pi_{(a,b)}$ of the form \eqref{eq:plane}. As in the proof of Theorem \ref{thm:degTCpolyparam}, this amounts to counting the number of solutions of the system
\begin{equation}\label{eq:systemratparam}
\frac{H_{10}(t)}{g_0(t)}+s \, \frac{H_{11}(t)}{g_0^2(t)} =0 , \ \frac{H_{20}(t)}{g_0(t)}+s \, \frac{H_{21}(t)}{g_0^2(t)} =0
\end{equation}
where
$H_{k0}(t):=\sum_{0\le i\le n}  a_{ki} g_i(t)$  and $H_{k1}(t):= \sum_{1\le j \le n} b_{kj}(g_j'(t)g_0(t)-g_j(t)g_0'(t))$  for $k=1,2$. Note that, for generic $(a,b)\in \A^{4n+2}$, all common solutions of the system  lie in the dense open set $\widehat{\mathcal{U}}\subseteq \A^2$. Furthermore, those $t_0\in K$ such that the system \eqref{eq:systemratparam} --which is linear in $s$-- has a solution whose first coordinate equals $t_0$ are zeros of
\[\Delta_{(a,b)}(t):=\det\left(\begin{array}{cc} H_{10}(t) & H_{11}(t)\\ H_{20}(t) & H_{21}(t) \end{array}\right).\]
For generic $(a,b)$, since the roots of $\Delta_{(a,b)}$ are not roots of $H_{11}$ or $g_0$, it follows that for each root $t_0$ of $\Delta_{(a,b)}$, there is a unique $s_0\in K$ such that $(t_0,s_0)$ is a solution of \eqref{eq:systemratparam}. Therefore, the degree of $T\CC$ is bounded by the degree of $\Delta_{(a,b)}$ for generic $(a,b) \in \A^{4n+2}$.

Now, for generic $(a,b)$, we have that $\deg(H_{k0})= \max\{\deg(g_i); 0\le i \le n\} = \deg(\CC)$ and $\deg(H_{k1}) = \max\{\deg(g_j' g_0-g_j g_0'); 1\le j \le n\}=  \max\{\deg(g_j); 1\le j \le n\} +\deg(g_0) - 1 \le 2\deg(\CC) -2$, because of our assumption that $\deg(g_0)>\deg(g_i)$ for $i=1,\dots, n$. We conclude that $\deg(\Delta_{(a,b)}) \le 3 \deg(\CC)-2$ and the result follows.
\end{proof}

\bigskip
Note that, for the particular class of curves having a (polynomial or rational) parametric representation, Theorems \ref{thm:degTCpolyparam} and \ref{thm:degTCratparam} give upper bounds for the degree of the tangent bundle that are linear on the degree of the curve. In particular, these bounds are sharper than the general bounds for the tangent bundle of arbitrary curves we obtain by applying the B\'ezout inequality (see Proposition \ref{proposition hipersup} for plane curves and Theorem \ref{cota principal} below for $d=1$ and arbitrary $n$).

\section{Bounds for the degree of the tangent bundle} \label{sec:bounds}

Throughout this section we discuss intrinsic bounds for the geometric degree of the tangent bundle of a smooth irreducible variety $V\subseteq \A^n$ of arbitrary dimension. 

\subsection{Upper bounds}\label{sec: upper_bounds}

For the case when $V$ is a hypersurface, Proposition \ref{proposition hipersup} provides an upper bound for $\deg(TV)$ as an immediate consequence of the B\'ezout theorem.  Now we will study the optimality of this bound.

\begin{example} \label{ejemplo optimalidad}
Let $\alpha, \beta \in K \setminus \{0\}$. Consider for each $(\alpha,\beta)$ the affine variety $W(\alpha,\beta) \subseteq \mathbb{A}^{2}$ given by
\begin{equation*}
    W(\alpha,\beta) := \{(x_1,x_2) \in \mathbb{A}^{2} \mid \alpha x_1^m + \beta x_2^m - 1 =0\}. 
\end{equation*}
It is easy to see that $W(\alpha,\beta)$ is an irreducible smooth curve. Therefore from Proposition \ref{prop:dimTV} its tangent bundle $TW(\alpha,\beta) \subseteq \mathbb{A}^{4}$, is a smooth irreducible algebraic variety of dimension $2$. Now, let $\Pi \subseteq \mathbb{A}^{4}$ be a generic linear variety of dimension $2$ given by
\begin{equation*}
    \Pi:= \left \{(x_1,x_2,y_1,y_2) \in \mathbb{A}^4 \mid  
     y_1 = a_1 x_1 + b_1 x_2 + c_1 , \
        y_2 = a_2 x_1 + b_2x_2 + c_2  
   \right \}.
\end{equation*}
For every $(\alpha,\beta)$, we will estimate $\deg(TW(\alpha,\beta))$ in terms of $m = \deg(W(\alpha,\beta))$. 
We have:
\begin{align*}
    \deg(TW(\alpha,\beta)) & = \#(TW(\alpha,\beta) \cap \Pi) \\ & = \#  \{  (x_1,x_2) \in \mathbb{A}^2 \mid 
    \alpha x_1^m+ \beta x_2^m=1,  \\
    & \hspace{1.8cm} m \alpha x_1^{m-1} (a_1 x_1 + b_1 x_2 + c_1)  + m \beta x_2^{m-1}(a_2 x_1 + b_2 x_2 + c_2) = 0   
     \}.
\end{align*}
Now, by the genericity of $a_i,b_i,c_i$, we can assume that all the solutions $(x_1,x_2)$ of the system lie in the torus $(\mathbb{A}^{1} \setminus\{0\})^{2}$. Then, we count this number of solutions by means of the Bernstein-Kushnirenko's theorem (see Proposition \ref{prop:BKK}):
for generic $(\alpha,\beta) \in \mathbb{A}^{2}$ we have a generic system of two polynomials in two variables such that the support of the first polynomial is $m\Delta^2$, where $\Delta^2$ is the $2$-standard simplex, and the support of the second polynomial is the trapezoid $T$ with vertices $(m-1,0), (m,0),$ $(0,m), (0,m-1)$. It follows that
 $\deg(TW(\alpha,\beta)) = MV(m \Delta^2, T)$.
By elementary computations, it can be shown that  $MV(m \Delta^2, T) = m^2$. Then, there exists $(\alpha,\beta)$ such that:
\begin{equation*}
    \deg(TW(\alpha,\beta)) = (\deg(W(\alpha,\beta)))^2.
\end{equation*}
\end{example}

By considering the product of the variety $W(\alpha,\beta)\subseteq \A^2$ with an appropriate affine space, one can show that the bound in  Proposition \ref{proposition hipersup} is optimal in the following sense: 

\begin{remark} \label{rem: optimality} For any $n,m \in \N$,  there exists a smooth irreducible hypersurface $V \subseteq \mathbb{A}^n$ of degree $m$ such that $\deg(TV) = m^2$.
\end{remark}

Unfortunately, the general case is rather different to the case of hypersurfaces. When applying Bézout's theorem directly, the obtained bounds for the degree of $TV$ depend on the number and degrees of polynomials defining $I(V)$.
However, by combining  Proposition \ref{prop:mumford} and B\'ezout's Theorem (see Proposition \ref{prop:bezout}) one can obtain the following intrinsic upper bound:

\begin{proposition}\label{ecuacion cota peor}
  Let $V \subseteq \mathbb{A}^n$ be a  smooth irreducible algebraic variety of dimension $d$. Then, $\deg(TV) \leq (\deg(V))^{n+d+1}$.  
\end{proposition}

\begin{proof}
From Proposition \ref{prop:mumford}, let  $f_1,\ldots,f_r$ in $K[\x]$ be a system of generators of $I(V)$ whose degrees are bounded by $\deg(V)$. By Propositon \ref{prop:dimTV}, $I(TV)$ can be generated by the polynomials $f_1,\ldots,f_r,\nabla f_1 \cdot \y, \ldots, \nabla f_r \cdot \y$ in $K[\x, \y ]$. Therefore, applying the first inequality in Proposition \ref{prop:bezout} to the variety $V\times \A^n$ and the hypersurfaces defined by polynomials $\nabla f_i \cdot \y$, $i=1,\dots, r$ in $\A^n \times \A^n$, we obtain
\begin{equation*}
\deg(TV)\le \deg(V\times \A^n)\left(\max_{1\le i\le r} \deg(\nabla f_i\cdot \y)\right)^{\dim(V \times \mathbb{A}^n)}\le \deg(V)^{1+d+n},
\end{equation*} 
since $\deg(V\times \A^n) = \deg(V)$, $\deg(\nabla f_i \cdot \y)\le \deg(V)$ for $i=1,\dots,r$, and $\dim(V\times \A^n)= d+n$.
\end{proof}

\bigskip
 Even though the previous bound is intrinsic, it is far from optimal: Proposition \ref{proposition hipersup} and Remark \ref{rem: optimality} show that in the case $d=n-1$, the degree of $TV$ is bounded optimally by $(\deg(V))^{2}$, whereas in this case the inequality in Proposition \ref{ecuacion cota peor}  gives  $\deg(TV) \leq (\deg(V))^{2n}$.

We now present the main result for this section, which gives a more precise, intrinsic upper bound for $\deg(TV)$.

\begin{theorem}\label{cota principal}
Let $V \subseteq \mathbb{A}^n$ be a smooth irreducible algebraic variety of dimension $d$. Then,
\begin{equation*}
    \deg(TV) \leq \min \left \{(\deg(V))^{n-d+1}, \deg(V)\left((n-d)(\deg(V)-1)+1\right)^{d} \right \}.
\end{equation*}
\end{theorem}

\begin{proof}[{Proof of Theorem \ref{cota principal} (Theorem B in the Introduction)}]
Consider, by Proposition \ref{prop:mumford}, a system  of generators $f_1,\ldots,f_r$ for the ideal $I(V)\subseteq K[\x]$ with $\deg(f_i) \leq \deg(V)$ for $i=1,\ldots,r$. As $V$ is smooth, by the Jacobian criterion, there is an $n-d$ minor of the associated Jacobian matrix that do not vanish identically on $V$. Without loss of generality, we may assume that it is the principal minor corresponding to the first $n-d$ rows and the first $n-d$ columns. Let $\mathcal{U} \subseteq V$ be the non-empty open set given by the complement of the zero set of this minor. Denote by $T\mathcal{U}$ the algebraic set $T\mathcal{U}:=  (\mathcal{U}\times \mathbb{A}^n) \cap TV$, which is a dense subset of $TV$. It suffices to estimate $\deg(T\mathcal{U})$.

 Consider $\ell_1\ldots,\ell_{2d}$ generic polynomials of degree 1 in $K[\x, \y]$ such that
$$\deg(TU) = \# \left (T\mathcal{U} \cap \left \{ (x,y) \in \mathbb{A}^{2n} \mid \ell_k(x,y) = 0 \ \forall\, 1 \leq k \leq 2d  \right \} \right ).$$
Then, $\deg(T\mathcal{U})$ equals the cardinality of the set
\begin{equation} \label{eq:degTU}
     \left\{
    (x,y)\in \mathbb{A}^{2n} \mid 
       x \in \mathcal{U}, \ 
       \nabla f_i(x) \cdot y = 0 \ \forall\, 1 \leq i \leq n-d, \
        \ell_k(x,y) = 0 \ \forall \, 1 \leq k \leq 2d
  \right  \} ,
\end{equation}
since $\deg(\mathcal{U}) = \deg(V)$ as $\mathcal{U}$ is a dense open subset.
Now, applying Bezout's inequality, since $\deg(\nabla f_i(\x)\cdot \y) \leq \deg(V)$ for every $i$,  we obtain:
\begin{equation*}
    \deg(T\mathcal{U}) \leq \deg(\mathcal{U}) (\deg(V))^{n-d} = (\deg(V))^{n-d+1}. 
\end{equation*}

It remains to be shown that the following inequality holds:
\begin{equation}\label{eq:bound2_degTU}
    \deg(T\mathcal{U}) \leq \deg(V)((n-d)(\deg(V)-1)+1)^{d}.
\end{equation}
Let $A(\textbf{x}) \in K[\x]^{n\times n}$ be the matrix given by:
$A(\textbf{x}) = \begin{pmatrix}
        J \\ \hline B
    \end{pmatrix}$, 
where $J$ is the $(n-d) \times n$ matrix given by the Jacobian of $f_1,\ldots, f_{n-d}$ and $B$ is the $d\times n$ matrix given by the coefficients of the monomials $y_1,\dots, y_n$ in $\ell_1\ldots,\ell_{d}$, and  let $E(\textbf{x}) \in K[\x]^{n\times 1}$ be defined as 
$E(\textbf{x}) = (0,\dots, 0, -\ell_1(\textbf{x},0), \dots, -{\ell_d}(\textbf{x},0))^t$. 
In order to express $\textbf{y}$ in terms of $\textbf{x}$ we consider the system 
\begin{equation}\label{ecuacion en demostración}
    A(\textbf{x})\, \textbf{y}^{t} = E(\textbf{x}).
\end{equation}
By the genericity of $\ell_1,\ldots,\ell_{d}$, we may assume that $A(x)$ is invertible for all $x \in \mathcal{U}$. By Cramer's rule, for every $1\le j \le n$, the $j$th coordinate of the unique solution to \eqref{ecuacion en demostración} is
$y_j= \dfrac{\det(A_{j,E})}{\det(A)}$,
where $A_{j,E}$ denotes the matrix obtained by replacing the $j$th column of $A$ with $E$.
Set $P_0(\x):=\det(A)$ and, for $1\le j \le n$, $P_j(\x):= \det(A_{j,E})$. Substituting $\y= \left(\frac{P_1(\x)}{P_0(\x)}, \dots, \frac{P_n(\x)}{P_0(\x)}\right)$  in the last $d$ linear equations defining the set \eqref{eq:degTU} and clearing denominators, we obtain polynomials
$Q_i(\x) = P_0(\x) \cdot \ell_{d+i}\left(\x, \frac{P_1(\x)}{P_0(\x)}, \dots, \frac{P_n(\x)}{P_0(\x)}\right)$, for $1\le i \le d,$
satisfying:
$$\deg(T\mathcal{U}) = \# \left ( \left\{x \in \mathbb{A}^n \mid x \in \mathcal{U},  \
          Q_i(x) = 0  \ \forall \, 1 \leq i \leq d \right \} \right ).$$
It can be easily seen that $\deg(P_0) = \deg(\det(A))\le  (n-d) (\deg(V) - 1)$ and, for $1\le j \le n$, $\deg(P_j)=\deg(\det(A_{j,E})) \leq (n-d)(\deg(V)-1) + 1$ and, as the polynomials $Q_i$ are linear combinations of $P_0$, $x_1 P_0, \dots, x_n P_0$, $P_1,\dots, P_n$, then
$\deg(Q_i) \leq (n-d)(\deg(V)-1)+1$ for all $1\le i \le d$.
Therefore, the desired bound \eqref{eq:bound2_degTU} follows by Bezout's inequality. \end{proof}

\bigskip
Observe that, for the case $d= n-1$,  Theorem \ref{cota principal} recovers exactly Proposition \ref{proposition hipersup}. However, despite recovering the optimal bound in this case, we will show that when the setup is generic, better estimates can be obtained.

Let $n \in \mathbb{N}$. For each $j \in \mathbb{N}$, we write $\CL_j$ for the topological space $(K_{j}[\x], \tau)$ where $K_j[\x]$ is the set of polynomials of degree at most $j$ in $n$ variables and $\tau$ is the Zariski topology in their coefficient vectors. 
The following lemma summarizes the properties of algebraic varieties defined by generic polynomials in this topological space (see, for instance, \cite{Masser1983}).

\begin{lemma}\label{lema VF}
Let $r,n \in \mathbb{N}$, $r<n$, and  consider $k_1,\ldots,k_r \in \mathbb{N}$. Then there exists a non-empty Zariski open set $\mathcal{U} \subseteq \CL_{k_1} \times \cdots \times \CL_{k_r}$ such that, for each $F =(f_1,\ldots,f_r) \in \mathcal{U}$, the algebraic set $V(F)\subseteq \A^n$ is a smooth and irreducible variety of dimension $n-r$, $I(V(F)) = (f_1,\ldots,f_r)$ and $\deg(V(F)) = \displaystyle \prod_{i=1}^r k_i$. $\square$
\end{lemma}

Now, as an immediate consequence of the definition of the tangent bundle from equations defining the variety \eqref{eq:defTV}, Bezout's inequality, and Lemma \ref{lema VF}, we deduce:

\begin{corollary}\label{corolario cuadrado}
     Let $r, n \in \N$, $r<n$, and consider $k_1,\ldots,k_r \in \mathbb{N}$. Then there exists a non-empty Zariski open set $\mathcal{U} \subseteq \CL_{k_1} \times \ldots \times \CL_{k_r}$ such that, for each $F =(f_1,\ldots,f_r) \in \mathcal{U}$,  the variety $V_F:=V(F)\subseteq \A^n$ is smooth and irreducible, and we have:
     \begin{equation*}
         \deg(TV_F) \leq (\deg(V_F))^{2}.
     \end{equation*}
\end{corollary}

\begin{proof}
Consider a non-empty Zariski open set $\mathcal{U}\subseteq \CL_{k_1} \times \ldots \times \CL_{k_r}$ as in Lemma \ref{lema VF}. Then, for $F=(f_1,\ldots,f_r)\in \mathcal{U}$, the variety $V_F:=V(F)$ is smooth, irreducible, has dimension $n-r$ and $\deg(V_F) = \displaystyle \prod_{i=1}^r k_i$. Now, by proposition \ref{prop:dimTV}, $I(TV_F) = (f_1,\ldots,f_r, \nabla f_1 \cdot \y,\ldots, \nabla f_r \cdot \y)$ and then, by Bézout's inequality we obtain: 
\begin{equation*}
    \deg(TV_F) \leq \prod_{i=1}^r \deg(f_i) \cdot \prod_{i=1}^r \deg( \nabla f_i \cdot \y) = (\deg(V_F))^{2} 
\end{equation*}
as stated.
\end{proof}

\bigskip

Corollary \ref{corolario cuadrado} shows that if a variety is defined by generic polynomials, then the degree of its tangent bundle is bounded by the square of its degree, thus recovering the bound in Proposition \ref{proposition hipersup} which holds for hypersurfaces. On the other hand, Theorem \ref{cota principal} implies that if $\CC\subseteq \A^n$, $n\ge 2$, is an arbitrary irreducible smooth curve, then
\begin{equation*}
    \deg(T\CC) \leq \deg(\CC)((n-1)(\deg(\CC)-1) +1) \leq  (n-1) (\deg(\CC))^2,
\end{equation*}
which is also quadratic in the degree of the curve. 
Unfortunately, whenever $V$ is neither a curve nor a hypersurface, the upper bounds obtained in Theorem \ref{cota principal} are at least cubic in $\deg(V)$, but we do not have evidence that they are sharp. This leads us to pose the following question:

\begin{question} \label{conj: quadratic} Is there a universal polynomial $P(n, \deg(V))$, quadratic in $\deg(V)$, such that the inequality
\begin{equation*}
    \deg(TV) \leq P(n,\deg(V))
\end{equation*}
holds for every smooth irreducible algebraic variety $V\subseteq \A^n$?
\end{question}

\subsection{Varieties with tangent bundle of minimal degree}

Proposition \ref{Prop: Lowerboundtrivial} provides a lower bound for the geometric degree of the tangent bundle of a smooth irreducible algebraic variety $V\subseteq \A^n$.
A question that arises is whether the inequality in that result is strict or not. This is easily answered by considering $V$ as a linear variety because, in this case, $TV$ is also linear and hence,  both $V$ and $TV$ have degree $1$. 

We will see that, in fact, the equality between $\deg(TV)$ and $\deg(V)$ is a property that characterizes linear varieties. For simplicity, we will split the proof of this result in two steps. 
The first one is to prove it for curves.
\begin{lemma}\label{Lemma: curvalineal}
Let $\mathcal{C} \subseteq \mathbb{A}^n$ be a smooth irreducible algebraic curve such that $\deg(T\mathcal{C}) = \deg(\mathcal{C})$. Then $\mathcal{C}$ is a line.
\end{lemma}

\begin{proof}
By Theorem \ref{thm:degTC}, we know that $
    \deg(T\mathcal{C}) = \deg(\mathcal{C}) + \omega(\mathcal{C}) \deg(\operatorname{Tan}(\mathcal{C}))$. Cancelling out, we obtain $\omega(\mathcal{C}) = 0$ and then, we conclude that $\mathcal{C}$ is a line.
\end{proof}

\bigskip

Now, we are ready to state and prove the main theorem of this section.

\begin{theorem} \label{thm:linear}
    Let $V \subseteq \mathbb{A}^n$ be a smooth irreducible algebraic variety. The equality \linebreak
     $\deg(TV) = \deg(V)$ holds if and only if $V$ is linear.
\end{theorem}

\begin{proof}
    We will just prove the ``only if'' part of the statement, as the converse follows straightforwardly.
    As we have already proven the result for curves in Lemma \ref{Lemma: curvalineal}, let $d = \dim(V) > 1$ and suppose that $V$ is non-linear. Since $V$ is irreducible and non-linear, then $\deg(V) > 1$. Consider, by Bertini's Theorem 
    \cite[Ch.~I, Th\'eor\`eme 6.6]{jouanolou1983}, linear polynomials $\ell_1,\ldots,\ell_{d-1}$ in $\mathbb{A}^n$ such that $\mathcal{C} := V \cap V(\ell_1,\ldots, \ell_{d-1})$  is a smooth irreducible curve with $\deg(\mathcal{C}) = \deg(V) > 1$ and $I(\mathcal{C}) = I(V) + (\ell_1,\ldots,\ell_{d-1})$.

    As $\mathcal{C}$ is not a line, it follows from Proposition \ref{Prop: Lowerboundtrivial} and Lemma \ref{Lemma: curvalineal} that $\deg(T\mathcal{C}) > \deg(\mathcal{C}) = \deg(V)$. On the other hand, from Proposition \ref{prop:dimTV}, it can be easily seen that $$T\mathcal{C} = TV \cap V( \ell_1(\textbf{x}),\ldots, \ell_{d-1}(\textbf{x}), \ell_1(\textbf{y}),\ldots, \ell_{d-1}(\textbf{y})).$$
    Now, by B\'ezout's inequality we obtain that $\deg(T\mathcal{C}) \leq \deg(TV)$, and we conclude that
           $\deg(V) = \deg(\mathcal{C}) < \deg(T\mathcal{C}) \leq \deg(TV),$ contradicting the assumption $\deg(TV) = \deg(V)$.
\end{proof}

\end{document}